\documentclass{amsart}
\usepackage{amsthm}
\usepackage{mathrsfs} %This one is needed for the $\mathscr{H}$ that gives us cursive handwriting.
\usepackage{amsmath}
\usepackage{graphicx}
\usepackage{amsfonts}
\usepackage{amssymb}
\usepackage{inputenc}
\usepackage[bookmarks, colorlinks=false, plainpages = false, citecolor = green, urlcolor = blue, filecolor = blue, linkcolor = blue]{hyperref}

\newcommand{\h}{\mathfrak{h}}
\newcommand{\g}{\mathfrak{g}}

\newcommand{\m}{\mathfrak{m}}
\newcommand{\bs}{\backslash}

\newcommand{\bd}{\begin{displaymath}}
\newcommand{\ed}{\end{displaymath}}

\newtheorem{lema}{Lemma}[section]
\newtheorem{prop}[lema]{Proposition}
\newtheorem{cor}[lema]{Corollary}
\newtheorem{teo}[lema]{Theorem}

\newtheorem*{T}{Theorem}

\theoremstyle{remark}

\newtheorem{remark}{\textbf{Remark}}[section]

\begin{document}

%\title{ISOMETRIC ACTIONS OF QUATERNIONIC SYMPLECTIC GROUPS}
\title{Isometric actions of quaternionic symplectic groups}
\author{Manuel Sedano-Mendoza}
\email{manuel.sedano@cimat.mx}

%\date{}

\maketitle

\begin{abstract}
Denote by $Sp(k,l)$ the quaternionic symplectic group of signature $(k,l)$. We study the deformation rigidity of the embedding $Sp(k,l) \times Sp(1) \hookrightarrow H$, where $H$ is either $Sp(k+1,l)$ or $Sp(k,l+1)$, this is done by studying a natural non-associative algebra $\m$ comming from the affine structure of $Sp(1) \backslash H$. We compute the automorphism group of $\m$ and as a consecuence of this, we are able to compute the isometry group of $Sp(1) \bs H$ at least up to connected components. Using these results, we obtain a uniqueness result on the structure of $Sp(1) \bs H$ together with an isometric left $Sp(k,l)$-action and classify its finite volume quotients up to finite coverings. Finally, we classify arbitrary isometric actions of $Sp(k,l)$ into connected, complete, analytic, pseudo-Riemannian manifolds admitting a dense orbit of dimension bounded by $\textrm{dim}(Sp(1) \bs H)$.
\end{abstract}

\section{Introduction}

\subsection{Isometries of homogeneous spaces}
Homogeneous spaces $L \bs H$, where $H$ is a Lie group and $L \subset H$ is a closed subgroup, arise naturally as the most symmetric manifolds when considered with rigid geometric structures, such as pseudo-Riemannian manifolds or more generally affine connections. So for example a right-invariant pseudo-Riemannian metric $\lambda$ on $L \bs H$ satisfies that
	\[ R_g : L \bs H \rightarrow L \bs H, \qquad R_g(L h) = L (hg) \]
is an isometry for every $g \in H$, that is $R(H) \subset Iso(L \bs H,\lambda)$. An important family of homogeneous spaces are the Riemannian symmetric spaces where the isotropy subgroup $L$ is compactly embedded in $H$ and is determined by the fixed points under an involutive automorphism. The following is a classical result
\begin{T}[\cite{He}]
If $H$ is connected, semisimple and acts efectively in a Riemannian symmetric space $(L \bs H, \lambda)$, then $H$ realizes the connected component of the identity in the isometry group, i.e. $Iso(L \bs H,\lambda)_0 = R(H)$.
\end{T}

In \cite{Oni2}, Onishchik studied the setting where $H$ is a compact Lie group and $K,L \subset H$ are closed subgroups so that $L$ acts transitively in $K \bs H$, and we have that $K \bs H \cong M \bs L$, where $M = L \cap K$. At the Lie algebra level, the pair $(\h,\mathfrak{k})$ is called an extension or enlargement of the pair $(\mathfrak{l}, \mathfrak{m})$, where $\mathfrak{m}$ and $\mathfrak{k}$ are the Lie algebras of $L$ and $M$, Onishchik classified such extensions in three general families and gave a complete list of such extensions when $\h$ is a compact simple Lie algebra. Onishchik's classification has been used to compute $Iso(M)_0$, where $M$ is a compact homogeneous space with special properties, such as isotropy irreducible spaces, normal homogeneous and naturally reductive spaces, see \cite{Reg} and \cite{w-z}, however in general it is still an open problem to compute such group. If we move away from the compact case even less is know, so for example consider $H$ a non-compact simple Lie group with an involutive automorphism $\theta$, in general $H^\theta$ is not compactly embedded in $H$ but $H^\theta \bs H$ always admits a structure of a symmetric pseudo-Riemannian manifold, in this case it is still unknown if $R(H) = Iso(H^\theta \bs H)_0$.

Rigidity phenomena coming from dynamics appear when we consider non-compact homogeneous spaces with simple Lie groups, so for example compact quotients of non-compact Riemannian symmetric spaces, without closed two-dimensional factors, are completely determined by its fundamental group; this is in deep contrast to its compact counterpart, this result is given by Mostow's rigidity theorem, see \cite{Rat}. More generally we have the following result due to Margulis:

\begin{T}[Margulis' superrigidity \cite{Z1}]
Let $H$ be a centerless semisimple Lie group without compact factors and $rank \geq 2$, $\Gamma$ an irreducible lattice of $H$ and $G$ a non-compact simple Lie group. If $\pi : \Gamma \rightarrow G$ is a homomorphism with Zariski dense image, then $\pi$ has continuous extension to a homomorphism $H \rightarrow G$.
\end{T}

Mostow and Margulis' rigidity results highlight the relationship between geometry and topology of homogeneous spaces with a semisimple isometry group without compact factors, these results were generalized by R. Zimmer to more general actions in his cocycle superrigidity theorem \cite{Z1}, in this result, the ``irreducibility'' condition is taken to be in the volume sense, that is, ergodic actions. Zimmer obtained a large number of results on rigidity behaviour of $G$-actions, where $G$ is a non-compact simple Lie group, so that in \cite{Z2} he made a general statement conjecturing that the only examples of such ergodic actions up to trivial modifications are the double quotients $K \backslash H / \Gamma$ together with a non-trivial continuous homomorphism $\rho: G \rightarrow H$ such that $K \subset H$ is a compact subgroup that commutes with $\rho(G)$ and $\Gamma \subset H$ is a lattice, $G$ acts in the double quotient via left multiplications of $\rho(G)$; such conjecture became to be known as Zimmer's program, for a survey on more precise statements and conjectures around Zimmer's program see \cite{Fi}. M. Gromov contributed to Zimmer's program in \cite{G} in a more geometric context where he introduced the notion of rigid geometric structures, a notion that includes pseudo-Riemannian structures, basically, rigidity of a geometric structure is the condition needed for the group of isometries to be a Lie group. In his paper, Gromov proved that in the presence of the $G$-action with a dense orbit, the Lie algebra $\mathcal{H}$ of Killing fields centralizing the $G$-action covers all directions of the manifold in an open dense subset, here this result is stated in Theorem \ref{T4.1}, so that the manifold is almost an homogeneous space.

There have been recent works on the understanding of the structure of the Lie algebra $\mathcal{H}$ centralizing the $G$-action, so that under some hypothesis on the dimension of the manifold $M$ it is possible to say exactly what Lie algebra $\mathcal{H}$ is and to integrate to a global action of the simply connected Lie group $H$. The understanding of the structure of $\mathcal{H}$ is carried out using representation theory of $G$ and how the possible $G$-submodules of $\mathcal{H}$ glues together to obtain the Lie algebra structure, however this has been only possible by fixing $G$ explicitly because representation theory changes drastically as we change $G$, so for example in \cite{Q-O1} and \cite{Q-O2} an extensive analysis of such isometric actions is given in the cases where $G$ belongs to the first two infinite families of unitary groups, namely, $G$ is $SO(p,q)$ or $U(p,q)$ respectively. In the present work, we study the third and last infinite family of unitary groups, namely, the group of quaternionic unitary matrices $Sp(p,q)$, thus completing the understanding of the isometric actions in the first non-trivial dimensions of the three infinite families of classical unitary groups. 

\subsection{Main Results}

For every pair of integeres $p,q \in \mathbb{N}$, let us consider $H = Sp(p,q)$ the group of linear symmetries of the quaternionic hermitian space $\mathbb{H}^{p,q}$. If $K = Sp(1)$ and $G$ is either $Sp(p-1,q)$ or $Sp(p,q-1)$, then there is a cannonical injective homomorphism $G \times K \hookrightarrow H$ as block-diagonal matrices (see (\ref{inclusion1}) and (\ref{inclusion2}) in section \ref{Def}), so that $(G \times K) \backslash H$ is a symmetric space with the pseudo-Riemannian structure induced by the bilinear form $B(X,Y) = \textrm{Re}\ tr(XY)$, furthermore we have a sequence of pseudo-Riemannian submersions $H \rightarrow K \backslash H \rightarrow (G \times K) \backslash H$. As $G$ commutes with the compact group $K$, for every $g \in G$ and $h \in H$, the corresponding left and right multiplications
	\[ L_g \circ R_h : K \bs H \rightarrow K \bs H, \qquad L_g \circ R_h (K \ h') = K \ (g h' h) \]
are isometries, that is
	\[ L(G)R(H) = \{L_g \circ R_h : g \in G, \ h \in H \} \]
is a subgroup of isometries of $K \bs H$.

\begin{teo}\label{T3.1}
The isometry group of $K \bs H$ has finitely many components and 
	\begin{displaymath}
		Iso(K \bs H)_0 \cong L(G)R(H),
	\end{displaymath}
where $Iso(K \bs H)_0$ is the connected component of the identity of $Iso(K \bs H)$, moreover the homomorphism
	\begin{displaymath}
		\begin{array}{rcl}
	G \times H & \rightarrow & Iso(K \bs H) \\
	(g,h) & \mapsto & L_g \circ R_{h^{-1}}
		\end{array}
	\end{displaymath}
has as a Kernel $\{\pm(e_G, e_H) \} \cong \mathbb{Z}_2$.
\end{teo}

We may observe that $K \backslash H$ is a simply connected, non-compact, non-symmetric space whenever $p,q \geq 1$, so that the previously discussed known techniques does not apply. Our approach to this problem is via the affine structure of $K \backslash H$ induced from the Levi-Civita connection, the affine structure is coded into a non-associative algebra structure in the vector space $\m = T_{K e} (K \backslash H)$. Such construction is standard when we take the Lie group $H$ as a pseudo-Riemannian manifold so that the corresponding affine structure is coded in the Lie algebra $\h$ of the left invariant vector fields in $H$. If $H$ is semisimple, the adjoint representation realizes all the automorphisms of $\h$ at least in the connected component of the identity and thus $Iso(H)_0 = L(H) R(H)$. In our case, we use the fact that the affine structure $\m$ is invariant under the action of the simple Lie group $G$ so that we are able to compute the automorphism group of the non-associative algebra $\m$, using representation theory of $G$. We observe that this method is in principle applicable to any homogeneous space $S \backslash T$, where $(S \times L) \backslash T$ is a pseudo-Riemannian symmetric space where $T$ and $L$ are simple or semisimple Lie groups, thus we have a strategy for computing a wide family of homogeneous spaces. As a consequence of the previous theorem, we can control the fundamental group of a pseudo-Riemannian manifold $M$ that has $X$ as its universal covering, at least up to finite index, and we obtain

\begin{teo}\label{T3.2}
If $G$ acts isometrically and faithfully in a finite volume pseudo-Riemannian manifold $X$ whose universal covering is isometric to $K \bs H$ such that the $G$-action lifts to an action commuting with the right $H$-action, then there is a lattice $\Gamma \leq H$, an automorphism $\rho \in Aut(G)$ and an isometric finite covering
	\begin{displaymath}
		K \backslash H / \Gamma \rightarrow X,
	\end{displaymath}
that is $\rho(G)$-equivariant, in particular $X$ is complete.
\end{teo}

As another consequence of the characterization of the automorphisms of the non-associative algebra $\m = \mathfrak{k} \backslash \h$, we obtain a uniqueness result in the possible embeddings $G \times K \hookrightarrow H$, modulo automorphisms of $H$ (Theorem \ref{TA.1}), therefore we get a uniqueness result on the structure of the diffeomorphism and isometry type of $K \bs H$

\begin{teo}\label{C3.3}
Consider $\varphi_0, \varphi : G \times K \hookrightarrow H$ two injective homomorphisms, then there exists an analytic diffeomorphism 
	\[ F : \varphi_0(K) \bs H \rightarrow \varphi(K) \bs H, \]
 such that if $\overline{h}$ is a pseudo-Riemannian metric in $\varphi(K) \bs H$ with $L(G) R(H) \subset Iso(\varphi(K) \bs H, \overline{h})$, then we can rescale $\overline{h}$ along the $G$-orbits and their orthogonal complements so that $F$ is an isometry when $\varphi_0(K) \bs H$ is considered with the metric induced from the bilinear form $B$.
\end{teo}

Consider a non-compact simple group $Sp(k,l)$. Our study of isometric $Sp(k,l)$-actions on general pseudo-Riemannian manifolds relies on Gromov and Zimmer's machinery, that is, we study the algebra $\mathcal{H}$ centralizing the $Sp(k,l)$-action and using representation theory of $Sp(k,l)$ on the first non-trivial dimensions we are able to completely characterize $\mathcal{H}$ consequently getting the following general result

\begin{teo}\label{main4}
Suppose $M$ is a connected, complete, finite volume, pseudo-Riemannian manifold that admits an analytic isometric action of $Sp(k,l)$ with a dense orbit, such that $dim(M) \leq n(2n+5)$, where $n = k+l \geq 3$, then either one of the following cases hold
	\begin{enumerate}
\item The universal covering of the manifold denoted by $\widetilde{M}$ is equivariantly isometric to the pseudo-Riemannian product $Sp(k,l) \times N$ for some  simply connected pseudo-Riemannian manifold $N$.

\item For $H$ either $Sp(k+1,l)$ or $Sp(k,l+1)$, there is a lattice $\Gamma \leq H$, an automorphism $\phi : G \rightarrow G$ and a $\phi(G)$-equivariant finite covering 
	\begin{displaymath}
		\widehat{M} \rightarrow M
	\end{displaymath}
for the double quotient pseudo-Riemannian manifold $\widehat{M} = Sp(1) \backslash H / \Gamma$. Moreover, the pseudo-Riemannian metric in $M$ can be rescaled along the foliation defined by the $Sp(k,l)$-orbits and the orthogonal complements so that the finite covering is a pseudo-Riemannian covering map.
	\end{enumerate}
\end{teo}

We observe that in order to have non-compactness of $Sp(k,l)$, we need $k,l \geq 1$ and so $n = k+1 \geq 2$. If $n = 2$ we have the exceptional isomorphism $\mathfrak{sp}(1,1) \cong \mathfrak{so}(4,1)$, but this case was already covered in general for isometric actions of non-compact orthogonal groups in  \cite{Q-O1}, thus our hypothesis on the lower bound $n \geq 3$. 

The structure of the paper as follows: in section \ref{Symplectic}, we give a brief recall of some basic facts on the structure and representation theory of complex and quaternionic symplectic Lie groups and how they interact with orthogonal maps on their first non-trivial representation, summarized in Corollary \ref{C1.4}. The last part of the section consists of the study of the deformation rigidity of the embedding $G \times K \hookrightarrow H$ summarized in Theorem \ref{TA.1}. In section \ref{Stiefel}, we study the pseudo-Riemannian structure of $K \bs H$, where we compute the automorphism group of the non-associative algebra associated to the affine structure of $K \bs H$ and prove the Theorems \ref{T3.1}, \ref{T3.2} and \ref{C3.3}. In section \ref{arbitrary} we study the Lie algebra structure of the centralizer algebra $\mathcal{H}$ using the representation theory of $Sp(k,l)$ obtained in section \ref{Symplectic} and prove Theorem \ref{main4}.

\subsection*{Acknowledgements}
I would like to thank Raul Quiroga-Barranco for proposing the problem and providing the main ideas that led to this work, I am very grateful for the many fruitful discussions with him during the preparation of this paper. I thank CONACyT and CIMAT for the scholarships provided during my PhD studies and the preparation of my thesis in which this work is based.

\section{Symplectic groups and algebras}\label{Symplectic}

\subsection{Structure theory for complex symplectic algebras}\label{Sec1.1}

Consider the canonical symplectic structure in $\mathbb{C}^{2n}$ given by , 
	\[ \omega(x,y) = x^t J y, \qquad J = \left(\begin{array}{cc} 0 & I_n \\ -I_n & 0 \end{array}\right), \]
where elements of $\mathbb{C}^{2n}$ are thought as column vectors and $I_n$ is the $n \times n$-identity matrix. The \textbf{complex symplectic group} $Sp(2n,\mathbb{C})$ is the subgroup of $GL_{2n}(\mathbb{C})$ preserving $\omega$ and its Lie algebra is the complex symplectic Lie algebra  characterized by the matrix relation $X J + J X^t = 0$, so it is given by
	\begin{displaymath}
		\mathfrak{sp}(2n,\mathbb{C}) = \left\{ \left(\begin{array}{cc} A & B \\ C & - A^t
		\end{array}\right) :  \begin{array}{c} A \in  M(n, \mathbb{C}), \\
					 B, C \ \textrm{symmetric} \end{array} \right\}.
	\end{displaymath}
A finite dimensional representation of $\mathfrak{sp}(2n,\mathbb{C})$ into a vector space $V$, also called a $\mathfrak{sp}(2n,\mathbb{C})$-module structure, is a Lie algebra homomorphism 
	\[ \mathfrak{sp}(2n,\mathbb{C}) \rightarrow \mathfrak{gl}(V), \]
and so, a linear action of the elements of $\mathfrak{sp}(2n,\mathbb{C})$ in $V$. If $V$, $W$ are two $\mathfrak{sp}(2n,\mathbb{C})$-modules, we say that a linear map $T: V \rightarrow W$ is a $\mathfrak{sp}(2n,\mathbb{C})$-homomorphism if $T(X \cdot v) = X \cdot T(v)$, for the corresponding actions of $X$ on $V$ and $W$ and every $X \in \mathfrak{sp}(2n,\mathbb{C})$, $v \in V$, so that an isomorphism of $\mathfrak{sp}(2n,\mathbb{C})$-modules is just a bijective homomorphism. A representation of $\mathfrak{sp}(2n,\mathbb{C})$ is called irreducible if it doesn't have non-trivial invariant subspaces and every representation $V$ decomposes as a direct sum $V = V_1 \oplus ... \oplus V_r$, where $V_j$ is $\mathfrak{sp}(2n,\mathbb{C})$-invariant and irreducible subspace, so that irreducible representations are the building blocks of representation theory.

Denote by $e_{i,j}$ the $n \times n$ matrix that has $1$ in the $(i,j)$-coordinate and 0 elsewhere and define
	\[
		\mathfrak{h} = \bigoplus_{i=1}^n \mathbb{C} X_{i,i}, \qquad 
		X_{i,j} := \left(\begin{array}{cc}
	e_{i,j} & 0 \\
	0 & - e_{j,i}
		\end{array}\right),
	\]
then $\h$ consists precisely of the diagonal matrices in $\mathfrak{sp}(2n,\mathbb{C})$. Consider $\{\varepsilon_j\}$ the basis of $\h^*$ such that $\varepsilon_j(X_{i,i}) = \delta_{ij}$, then the elements $\varepsilon_j$ are the weights of $\mathbb{C}^{2n}$, i.e. the simultaneous spectrum of the action of $\h$, more precisely, if $\{e_j, e^j\}$ is the canonical basis of $\mathbb{C}^{2n}$, then
	\[ X e_j = \varepsilon_j(X) e_j, \qquad X e^j = -\varepsilon_j(X) e^j, \qquad \forall \ X \in \h, \ 1 \leq j \leq n. \]
In general, $\h$ always acts by simultaneously diagonalizable operators in every representation, where the simultaneous spectrum is encoded in the weights of the representation that are elements of $\h^*$ and every irreducible representation is determined by a unique weight called the highest weight of the representation. To be more precise, define the elements
	\begin{displaymath}
		\omega_k = \varepsilon_1 + ... + \varepsilon_k, \quad 1 \leq k \leq n,
	\end{displaymath}
called simple weights and consider the dominant weight lattice of $\mathfrak{sp}(2n,\mathbb{C})$
	\begin{displaymath}
		\Gamma^d = \left\{k_1 \omega_1 + ... + k_n \omega_n : k_1, ... , k_n \in \mathbb{N} \cup \{0\} \right\} \subset \h^*,
	\end{displaymath}
then the complex irreducible representations of $\mathfrak{sp}(2n,\mathbb{C})$ are in 1-1 correspondence with the elements of $\Gamma^d$ and an irreducible representation with highest weight $\lambda \in \Gamma^d$ is denoted by $V(\lambda)$. For details and further properties in the structure theory and classification of representations of $\mathfrak{sp}(2n,\mathbb{C})$, see \cite{G-W} or \cite{Kn}.

\begin{remark}\label{examples}
Here we give explicit descriptions of the $\mathfrak{sp}(2n,\mathbb{C})$-representations associated to the simple weights
	\begin{enumerate}
		\item If we take $V = \mathbb{C}^{2n}$, then the canonical basis $\{e_j, e^j : 1 \leq j \leq n\}$ is a basis of weight vectors, where $e_j$ and $e^j$ have corresponding weights $\varepsilon_j$ and $-\varepsilon_j$ respectively, moreover $\varepsilon_1$ is the highest weight so we have $\mathbb{C}^{2n} \cong V(\omega_1)$.

		\item If we take the kth-alternating power $V = \Lambda^k \mathbb{C}^{2n}$, then there is a $\mathfrak{sp}(2n,\mathbb{C})$-homomorphism
			\[ \varphi_k : \Lambda^k \mathbb{C}^{2n} \rightarrow \Lambda^{k-2} \mathbb{C}^{2n}, \]
that can be defined as a contraction using the symplectic form $\omega$, for the precise definition see \cite{F-H}, such that $V(\omega_k) = Ker(\varphi_k)$. As a
consequence of this we have the decomposition $\Lambda^k \mathbb{C}^{2n} \cong V(\omega_k) \oplus \Lambda^{k-2} \mathbb{C}^{2n}$ and an explicit formula for its dimension 
	\[ \textrm{dim}(V(\omega_k)) = {2n \choose k} - {2n \choose k - 2}. \]
As a particular case, we have $\Lambda^2 \mathbb{C}^{2n} \cong V(\omega_2) \oplus \mathbb{C} \theta$, where $\theta = \sum_{j=1}^n e_j \wedge e^j$ and
	\[ V(\omega_2) = \{x \wedge y \in \Lambda^2 \mathbb{C}^{2n} : \omega(x,y) = 0 \}. \]

		\item If we take the symmetric power $V = S^2 \mathbb{C}^{2n}$, then 
	\[ S^2 \mathbb{C}^{2n} \rightarrow \mathfrak{sp}(2n,\mathbb{C}), \qquad x \odot y \mapsto (x y^t + y x^t) J \]
is an isomorphism of $\mathfrak{sp}(2n,\mathbb{C})$-representations, where $\mathfrak{sp}(2n,\mathbb{C})$ is considered with the adjoint representation. Moreover $e_1 \odot e_1$ is a highest weight vector, with weight $2 \omega_1$, so we have $S^2 \mathbb{C}^{2n} = V(2 \omega_1) \cong \mathfrak{sp}(2n,\mathbb{C})$.
	\end{enumerate}
\end{remark}
%---------------------------------------------------------------------------------------------------------------
The following Lemma gives us a first estimation on dimensions of representations of $\mathfrak{sp}(2n,\mathbb{C})$.
\begin{lema}\label{L1.4}
If $n > 3$ and $3 \leq k \leq n$, then 
	\begin{displaymath}
		\textrm{dim}(V(\omega_k)) > n(2n + 1) = \textrm{dim}(\mathfrak{sp}(2n,\mathbb{C})).
	\end{displaymath}
\end{lema}

\begin{proof}
Denote $D_k^n = {2n \choose k} - {2n \choose k - 2}$, the inequality can be verified directly in the first cases so that we can suppose $n \geq 6$, and by observing that the function
	\begin{displaymath}
		f(n) = D_3^n - n(2n + 1) = \frac{4}{3} n \left( n - \frac{7}{2} \right) \left(n + \frac{1}{2}\right)
	\end{displaymath}
is always positive when $n > 3$, we can further suppose $k \geq 4$. We have the expresion
	\begin{displaymath}
		D_k^n = \frac{2n(2n - 1) \cdots (2n - k + 3)}{k!} \left( 4 n^2 - 4nk + 6n - 2k + 2 \right),
	\end{displaymath}
and an immediate inequality 
	\begin{displaymath}
		\frac{n}{2}(4 n^2 - 4nk + 6n - 2k + 2) \geq \frac{n}{2}(4 n^2 - 4n^2 + 6n - 2n + 2) = n(2n +1),
	\end{displaymath}
that follows from the condition $n \geq k$, thus we only need to show that 
	\begin{displaymath}
		\frac{4(2n-1)(2n-2)\cdots(2n-k+3)}{k!} > 1,
	\end{displaymath}
so for example for $k= 4$ and $k=5$ we have
	\begin{displaymath}
		\frac{4(2n-1)}{4!} \geq \frac{7}{6}, \quad \frac{4(2n-1)(2n-2)}{5!} \geq \frac{12}{5}.
	\end{displaymath}
For every $n \geq k > 5$ we have
	\begin{displaymath}
		(2n - k + 5)(2n - k + 4)(2n - k + 3) \geq (n + 5)(n + 4)(n + 3) > 6!, 
	\end{displaymath}
so by adding $k - 6$ factors of the form $(2n-k+s) > s+1$, we obtain
	\begin{displaymath}
		(2n-1)\cdots(2n-k+3) > k!,
	\end{displaymath}
and the result follows.
\end{proof}

In general, the dimension of $V(\lambda)$ can be computed using Weyl's dimension formula \cite[Ch. 7, Exer. 10]{G-W} to obtain 
	\begin{equation}\label{WDF}
		\begin{array}{rcl}
		\textrm{dim}(V(\lambda)) & = & \prod_{1 \leq i < j \leq n} \left\{ 1 + \frac{m_i + \cdots + m_{j-1}}{j-i} \right\} \\
	& \times & \prod_{1 \leq i < j \leq n} \left\{ 1 + \frac{m_i + \cdots + m_{j-1} + 2(m_j + \cdots + m_n) }{2n + 2 - j - i } \right\} \\
	& \times & \prod_{1 \leq i \leq n} \left\{ 1 + \frac{m_i + \cdots + m_n}{n+1-i} \right\},
		\end{array}
	\end{equation}
where $\lambda = m_1 \omega_1 + \cdots m_n \omega_n $, using this formula we get an estimation on the possible representations under the dimension of the algebra.

\begin{cor}\label{C1.2}
Take $\lambda \in \Gamma^d$ a highest weight such that
	\[ \textrm{dim}(V(\lambda)) \leq \textrm{dim}(\mathfrak{sp}(2n,\mathbb{C})) = n(2n+1) \]
then $\lambda = \omega_1$, $\omega_2$ or $2 \omega_1$ if $n \neq 3$. If $n = 3$, $\lambda$ can be all the previous cases plus
	\begin{displaymath}
		\textrm{dim}(V(\omega_2)) =  \textrm{dim}(V(\omega_3)) < n(2n+1).
	\end{displaymath}
\end{cor}

\begin{proof}
We can see directly from $(\ref{WDF})$ that the dimension is strictly increasing as a function of weights and if $n \leq 3$, we can verify directly using $(\ref{WDF})$ that
	\[	\textrm{dim}(V(\omega_i + \omega_j)) > n(2n+1), \]
unless $i = j = 1$. If $n > 3$, Lemma \ref{L1.4} excludes the fundamental weights $\omega_k$, for $k \geq 3$. This reduces the calculations to weights that are linear combinations of $\omega_1$ and $\omega_2$, but 
	\begin{displaymath}
		\textrm{dim}(V(2 \omega_2)) = \frac{n(n-1)(2n-1)(2n+3)}{3} > n(2n + 1)
	\end{displaymath}
and
	\begin{displaymath}
		\textrm{dim}(V(\omega_1 + \omega_2)) = \frac{8 n (n-1)(n+1)}{3} > n(2n + 1),
	\end{displaymath}
so the only posibilities are $k \omega_1$ and $\omega_2$, given that $2 \omega_1$ already has the dimension of $\mathfrak{sp}(2n,\mathbb{C})$, the result follows.
\end{proof}

We add a technical Lemma that we will need in the next section

\begin{lema}\label{L1.1}
The following is an isomorphism of $\mathfrak{sp}(2n,\mathbb{C})$-modules
	\[	\Lambda^2 V(\omega_2) \cong V(2 \omega_1) \oplus V(\omega_1 + \omega_3). \]
\end{lema}

\begin{proof}
Observe that the elements
	\begin{displaymath}
		\sum_{j = 2}^n (e_1 \wedge e_j) \wedge (e_1 \wedge e^j),\ (e_1 \wedge e_2) \wedge (e_1 \wedge e_3) \in \Lambda^2 V(\omega_2) 
	\end{displaymath}
are highest weight vectors with highest weights $2 \omega_1$ and $\omega_1 + \omega_3$ respectively, therefore the corresponding highest weight modules appear as submodules of $\Lambda^2 V(\omega_2)$ (see \cite{Kn}[Sec. 5.3] for the characterization of highest weight vectors and a discussion of this property), i.e. we have
	\begin{displaymath}
		V(2 \omega_1) \oplus V(\omega_1 + \omega_3) \subseteq \Lambda^2 V(\omega_2).
	\end{displaymath}
Using $(\ref{WDF})$ we compute the dimension
	\[  \textrm{dim}(V(\omega_1 + \omega_3) = \frac{(n+1)(2n+1)(2n-1)(n-2)}{2}, \]
and from Remark \ref{examples}, we get 
	\[ \textrm{dim}(V(2 \omega_1)) = n(2n+1), \qquad \textrm{dim}(\omega_2) = n(2n - 1) - 1, \]
so a straighforward computation shows
	\[	\textrm{dim}(V(2 \omega_1)) + \textrm{dim}(V(\omega_1 + \omega_3)) = \textrm{dim}(\Lambda^2 V(\omega_2)),	\]
and the result follows. 
\end{proof}

\subsection{Quaterionic Symplectic Algebras}

For every $k,l \in \mathbb{N}$ and $n = k+l$, we consider $\mathbb{C}^{2k,2l}$ to be $\mathbb{C}^{2n}$ with an hermitian form of signature $(2k,2l)$ and add a compatible symplectic structure in the folllowing way:  take $\{e_s, e^s \}$ a basis of $\mathbb{C}^{2n}$ so that the hermitian form and symplectic form are given respectively by
			\[	\langle x, y \rangle =  \sum_{s=1}^n \varepsilon_s (\overline{x_s} y_s + \overline{x^s} y^s ), \qquad \omega( x, y )=  \sum_{s=1}^n \varepsilon_s (x_s y^s - x^s y_s )\]
where $x = \sum_s x_s e_s + x^s e^s$, $y = \sum_s y_s e_s + y^s e^s$ and
	\[ \varepsilon_s = \langle e_s , e_s \rangle = \langle e^s , e^s \rangle = \left\{\begin{array}{rcl}
	1, & & 1 \leq s \leq k, \\
	-1, & & k+1 \leq s \leq n.
 \end{array}\right.
	\]
We define the quaternionic symplectic group of signature $(k,l)$ to be the linear group preserving both the hermitian and symplectic structures in $\mathbb{C}^{2k,2l}$, i.e.
	\[ Sp(k,l) = \left\{A \in M_{2n \times 2n}(\mathbb{C}) \ : \ \begin{array}{c} \langle Ax , Ay \rangle = \langle x , y \rangle, \quad \omega( Ax , Ay ) = \omega( x , y ), \\ 	\forall \ x, y \in  \mathbb{C}^{2k,2l}	\end{array} \right\}, \]
if either $k$ or $l$ is zero, then we denote the group simply by $Sp(n)$.

\begin{remark}
There is a quaternionic structure in $\mathbb{C}^{2k,2l}$ given by a conjugate-linear map $j : \mathbb{C}^{2k,2l} \rightarrow \mathbb{C}^{2k,2l}$ that satisfies $j(e_s) = e^s$, $j(e^s) = - e_s$. We have thus the identification $\mathbb{C}^{2k,2l} \cong \mathbb{H}^{k,l} = \mathbb{C}^{k,l} \oplus j \mathbb{C}^{k,l}$ and the quaternionic hermitian structure 
	\[ \sum_{s = 1}^n \varepsilon_s (x_s + j x^s)^* (y_s + j y^s) = \langle x, y \rangle + j \omega(x,y), \]
where $z \mapsto z^*$ denotes conjugation in the quaternions, that induces the isomorphism of Lie groups
	\[ Sp(k,l) \cong \{ A \in M_{n \times n} (\mathbb{H}) : A^* I_{k,l} A = I_{k,l} \}, \quad
	I_{k,l} = \left(\begin{array}{cc} I_k & 0 \\ 0 & - I_l \end{array} \right).
 \]
\end{remark}

\begin{prop}\label{P1.7}
The group $Sp(k,l)$ is connected and simply connected with center \[ Z(Sp(k,l)) = \{\pm I \} \cong \mathbb{Z}_2, \] moreover $Sp(k)$ is compact and $Sp(k,l)$ is non-compact if $k,l \geq 1$.
\end{prop}

\begin{proof}
If we consider the identification $M_{n \times n}(\mathbb{H}) \cong \mathbb{H}^n \times \cdots \times \mathbb{H}^n = \mathbb{H}^{n^2}$, where each $\mathbb{H}^n$-factor is thought as the corresponding column, then we have an embedding of topological spaces
	\[ Sp(n) \subset \mathbb{S}(n) \times ... \times \mathbb{S}(n), \]
with closed image, where $\mathbb{S}(n) \subset \mathbb{H}^{n}$ is the unit sphere. This tells us that $Sp(n)$ is compact for every $n \in \mathbb{N}$, moreover, there is a fibration
	\begin{displaymath}
		Sp(n) \hookrightarrow Sp(n+1) \rightarrow \mathbb{S}(n+1) \cong Sp(n+1) / Sp(n),
	\end{displaymath}
given by the natural left multiplication so that by cellular aproximation (see \cite{Ha}) it is possible to see that $\pi_s (\mathbb{S}(n+1)) = 0$, for every $s < 4n + 3$, and by applying sucessively the long homotopy sequence we get that
	\begin{displaymath}
		\pi_s(Sp(n)) \cong \pi_s(Sp(1)) = 0, \qquad n \geq 1, \ s \in \{ 0,1 \},
	\end{displaymath}
where the last identity follows from the fact that $Sp(1) = \mathbb{S}(1)$. In general, observe that the automorphism given by $\sigma(X) = I_{k,l} X I_{k,l}$ is a Cartan involution and so, the group $K = Sp(k,l)_\sigma = Sp(k) \times Sp(l)$ is a maximally compact subgroup. We have thus Cartan decompositions of Lie algebra and Lie group \cite[Pg. 361]{Kn}
	\begin{displaymath}
		\mathfrak{sp}(k,l) = \mathfrak{k} \oplus \mathfrak{p}, \qquad Sp(k,l) \cong K \times \mathfrak{p},
	\end{displaymath}
where the second congruence is in the diffeomorphism type, again we have $\pi_s(Sp(k,l)) = \pi_s(Sp(k)) \times \pi_s(Sp(l)) = 0$, for $s \in \{ 0,1\}$, and thus $Sp(k,l)$ is connected and simply connected. In the case where $k,l \geq 1$, there is an unbounded one-parameter subgroup 
 	\[ \beta(t) = \left(\begin{array}{ccc} I_{k-1} & 0 & 0 \\
		0 & A(t) & 0 \\	0 & 0 & I_{l-1} \end{array}\right) \in Sp(k,l) \]
given by 
	\[ A(t) = \left(\begin{array}{cc} \cosh(t) & \sinh(t) \\ \sinh(t) & \cosh(t) \end{array}\right), \]
so in this case, $Sp(k,l)$ is non-compact. To compute the center of the groups, observe that  $Sp(n)$ has diagonal matrices of arbitrary diagonal elements $\lambda_j \in \mathbb{H}$ such that $|\lambda|^2 = 1$, so that if $h \in Z(Sp(n))$, then $h$ must be diagonal with diagonal elements given by unitary quaternions that belong to $Z(\mathbb{H}^*) = \mathbb{R}^*$, thus $h$ is diagonal with diagonal elements $\pm 1$. If $n \geq 2$, take $\alpha, \beta \in \mathbb{C}$ two distinct complex numbers that satisfy $|\alpha|^2 + |\beta|^2 = 1$, then 
	\[ \left(\begin{array}{ccc} \alpha & \beta & 0 \\
		- \overline{\beta} & \overline{\alpha} & 0 \\
		0 & 0 & I_{n-2}
	\end{array}\right) \in Sp(n), \]
must commute with $h$ that forces the $(1,1)$-entry of $h$ to be distinct to the $(2,2)$-entry, analogously, we can compare the remaining diagonal entries so that we see that the only possibility is $h = \pm I$ and thus $Z(Sp(n)) \cong \mathbb{Z}_2$. In general signature, we take again the previous Cartan decomposition so that the maximally compact subgroups always contain the center \cite[Thm. 6.31]{Kn} and thus $Z(Sp(k,l)) \subset Z(Sp(k)) \times Z(Sp(l)) \cong \mathbb{Z}_2 \times \mathbb{Z}_2$. Finally, if $\alpha_0, \beta_0 \in \mathbb{C}$ are two complex numbers such that $|\alpha_0|^2 - |\beta_0|^2 = 1$ and $\beta_0 \neq 0$, then 
	\[ \left(\begin{array}{cccc} \alpha_0 & 0 & \beta_0 & 0 \\	0 & I_{k-1} & 0 & 0 \\	- \overline{\beta}_0 & 0 & \overline{\alpha}_0 & 0 \\ 0 & 0 & 0 & I_{l-1}
 \end{array}\right)  \in Sp(k,l) \]
does not commute with
	\begin{displaymath}
		h = \left(\begin{array}{cc}
		\pm I & 0 \\ 0 & \pm I
	\end{array}\right) \in Sp(k,l)
	\end{displaymath}
if $h$ is not a multiple of the identity, so again $Z(Sp(k,l)) \cong \mathbb{Z}_2$. 
\end{proof}
 
Recall that if $\g_0$ is a real Lie algebra, then we say that $\g_0$ is a real form of a complex Lie algebra, or that $\g$ is the complexification of $\g_0$ if $\g_0 \otimes \mathbb{C} \cong \g$. Moreover, the highest weight classification of complex representations of $\g$, gives a classification of real representations of $\g_0$ in the following way: For every real irreducible representation $W$ of $\g_0$, there exist $V(\lambda)$ a complex irreducible representation of $\g$ with highest weight $\lambda$ such that either
		\begin{itemize}
	\item $W \otimes \mathbb{C} \cong V(\lambda)$ and we say that $W$ is a real form of $V(\lambda)$, or 

	\item $W = V(\lambda)_\mathbb{R}$, i.e. we take $V(\lambda)$ as a real vector space by forgetting the multiplication by complex scalars, we say that $W$ is the realification of $V(\lambda)$;
		\end{itemize}
in either case, we say that $W$ is a real representation of highest weight $\lambda$. If we define homomorphisms between real representations as in the complex case, we can observe that the image and kernel of every homomorphism will be an invariant subspace, so we have the following result known as Schur's Lemma

\begin{lema}\label{Schur}
If $T : V \rightarrow W$ is a homomorphism between two irreducible representations, then either $T = 0$ or $T$ is an isomorphism.
\end{lema}

\begin{prop}\label{real-forms}
There exist a decomposition 
	\[ \mathfrak{su}(2k,2l) = \mathfrak{sp}(k,l) \oplus W_0, \]
that is stable under the adjoint action of $\mathfrak{sp}(k,l)$ such that $W_0 \otimes \mathbb{C} \cong V(\omega_2)$ as a complex representation of $\mathfrak{sp}(k,l) \otimes \mathbb{C} \cong \mathfrak{sp}(2n,\mathbb{C})$.
\end{prop}

\begin{proof}
Observe that if $x, y \in \mathbb{C}^{2k,2l}$ are considered as column vectors, then 
	\[	\omega(x,y) = x^t J_{k,l} y, \qquad  J_{k,l} = \left(\begin{array}{cc} 0 & I_{k,l} \\ - I_{k,l} & 0 \end{array}\right). \]
Now the map $\sigma : \mathfrak{gl}(2n,\mathbb{C}) \rightarrow \mathfrak{gl}(2n,\mathbb{C})$ given by $\sigma (X) = J_{k,l} X^t J_{k,l}$ is involutive ($\sigma^2 = id$) and induces automorphisms in $\mathfrak{su}(2k,2l)$ and $\mathfrak{sl}(2n,\mathbb{C})$. These automorphisms induce corresponding decompositions into $(\pm 1)$-eigenvector subspaces
	\[ \mathfrak{su}(2k,2l) = \mathfrak{sp}(k,l) \oplus W_0, \qquad \mathfrak{sl}(2n,\mathbb{C}) = \g \oplus W, \]
where $\g$ is conjugated to $\mathfrak{sp}(2n,\mathbb{C})$ in $\mathfrak{sl}(2n,\mathbb{C})$. Moreover the decomposition in $\mathfrak{su}(2k,2l)$ complexifies into the one in $\mathfrak{sl}(2n,\mathbb{C})$, so to conclude the proof, we need to show that $W \cong V(\omega_2)$ as a representation of $\mathfrak{sp}(2n,\mathbb{C})$. To see this, observe that the map 
	\begin{equation}S : \Lambda^2 (\mathbb{C}^{2k,2l}) \rightarrow \mathfrak{gl}(2n,\mathbb{C}), \qquad 
		 S_{x \wedge y} (z) = (x^t J_{k,l} z) y - (y^t J_{k,l} z) x, \end{equation}
is a $\g$-homomorphism so that $\textrm{tr}(S_{x \wedge y}) = 2 \omega(x,y)$ and $\sigma(S_\eta) = - S_\eta$. Thus, we have
	\[	V(\omega_2) \cong \{ x \wedge y \in \Lambda^2 (\mathbb{C}^{2k,2l}) : \omega(x,y) = 0 \} \cong W, 	\]
where the last isomorphism is due to Lemma \ref{Schur} applied to $S$.
\end{proof}

There is an identification $\mathbb{C}^{2k,2l}_\mathbb{R} \cong \mathbb{R}^{4k,4l}$ of real representations of $\mathfrak{sp}(k,l)$, with the invariant inner product given by $(X,Y) = \textrm{Re} \langle X, Y \rangle$. Thus we have successive inclusions of Lie groups
	\[ Sp(k,l) \subset SU(2k,2l) \subset SO(4k,4l), \]
and the following Lemma gives us the complete algebra of orthogonal operators in $\mathbb{R}^{4k,4l}$ in terms of the symplectic ones.

\begin{lema}\label{L1.6}
There is a $\mathfrak{sp}(k,l)$-module decomposition
	\begin{displaymath}
		\mathfrak{so}(4k,4l) = \mathfrak{sp}(k,l) \oplus \mathfrak{sp}(1) \oplus V_0 \oplus V_1 \oplus V_2
	\end{displaymath}
where $\mathfrak{sp}(1)$ is a subalgebra of $\mathfrak{so}(4k,4l)$ that commutes with $\mathfrak{sp}(k,l)$. For every $s \in \{0,1,2\}$, $V_s \otimes \mathbb{C} \cong V(\omega_2)$ as $\mathfrak{sp}(k,l)$-modules and 
	\begin{displaymath}
		[V_s , V_s] = [\mathfrak{sp}(k,l) , \mathfrak{sp}(k,l)] = \mathfrak{sp}(k,l),
	\end{displaymath}
where $[\cdot , \cdot]$ is the Lie bracket in $\mathfrak{so}(4k,4l)$.
\end{lema}

\begin{proof}
For every $s = 0,1,2$, consider the maps
	\[ T^s : \Lambda^2 \mathbb{C}^{2k,2l} \rightarrow \mathfrak{so}(4k,4l), \qquad T_{x \wedge y}^s(z) = B_s(x,z) y - B_s(y,z) x,  \]
where $B_1(x,y) = B_0(y,x) = \langle y , x \rangle$, and $B_2(x,y) = (x,y) = \textrm{Re}\langle x, y \rangle$, then $T^2$ induces the isomorphism $\Lambda^2 (\mathbb{C}^{2k,2l}_\mathbb{R}) \cong \mathfrak{so}(4k,4l)$ as $\mathfrak{so}(4k,4l)$-modules. Observe that for every $x,y,z,w \in \mathbb{C}^{2k,2l}$,
	\[ \langle T_{x \wedge y}^0(z) , w \rangle = - \langle z, T_{x \wedge y}^0(w) \rangle, \qquad 
		\langle T_{x \wedge y}^1(z) , w \rangle = - \langle T_{x \wedge y}^1(w) , z  \rangle , \]
so $T_{x \wedge y}^0 \in \mathfrak{u}(2k,2l) \subset \mathfrak{so}(4k,4l)$ and $\textrm{Im}(T^0) \cap \textrm{Im}(T^1) = 0$, moreover $T^s$ induce real $\mathfrak{u}(2k,2l)$-homomorphisms so that Lemma \ref{Schur} implies that $T^1$ is injective in $(\Lambda^2 \mathbb{C}^{2k,2l})_\mathbb{R}$ and $\mathfrak{u}(2k,2l) = \textrm{Im}(T^0)$. This gives us the decomposition
	\[ \mathfrak{so}(4k,4l) = \mathfrak{u}(2k,2l) \oplus W, \qquad W = \textrm{Im}(T^1) \cong (\Lambda^2 \mathbb{C}^{2k,2l})_\mathbb{R} \]
as $\mathfrak{u}(2k,2l)$-modules. Recall that $\Lambda^2 \mathbb{C}^{2k,2l} = V(\omega_2) \oplus \mathbb{C} \theta$ as $\mathfrak{sp}(2n,\mathbb{C})$-modules, for 
	\[ \theta = \sum_{s = 1}^n \varepsilon_s e_s \wedge e^s, \]
and Proposition \ref{real-forms} tells us that there exists a real $\mathfrak{sp}(k,l)$-representation $W_0$ such that $W_0 + i W_0 = V(\omega_2)$ and $\mathfrak{su}(2k,2l) = \mathfrak{sp}(k,l) \oplus V_0$, where $V_0 = S(W_0)$ and
	\[ S : \Lambda^2 \mathbb{C}^{2k,2l} \rightarrow \mathfrak{gl}(2n,\mathbb{C}), \qquad
		S_{x \wedge y}(z) = \omega(x,z)y - \omega(y,z)x
	\]
is a $\mathfrak{sp}(2n,\mathbb{C})$-homomorphism. Observe that the elements
	\[ 	\left(\begin{array}{cc} 0 & I_{2n} \\ - I_{2n} & 0 \end{array}\right), \quad 
	\left(\begin{array}{cccc} 0 & -I_n & 0 & 0 \\ I_n & 0 & 0 & 0  \\ 0 & 0 & 0 & I_n \\ 0 & 0 & -I_n & 0\end{array}\right), \quad
		\left(\begin{array}{cccc} 0 & 0 & 0 & I_n \\0 & 0 & -I_n & 0 \\ 0 & I_n & 0 & 0 \\ -I_n & 0 & 0 & 0\end{array}\right),
	\]
that represent $z \mapsto i z$, $T_\theta$ and $i T_\theta$ respectively in the basis $\{ e_s, e^s, ie_s, ie^s \}$ of $\mathbb{C}^{2k,2l}_\mathbb{R}$, generate a Lie algebra isomorphic to $\mathfrak{sp}(1)$ commuting with $\mathfrak{sp}(k,l)$. Thus we have the decomposition 
	\[ \mathfrak{so}(4k,4l) = \mathfrak{sp}(k,l) \oplus \mathfrak{sp}(1) \oplus V_0 \oplus V_1 \oplus V_2, \]
where $V_0 = S(W_0)$, $V_1 = T^1(W_0)$, $V_2 = T^1(i W_0)$ and thus $V_s \otimes \mathbb{C} \cong V(\omega_2)$ as complex representations of $\mathfrak{sp}(k,l)$. Consider now $s = k+1$ if $k,l \geq 1$ and $s = 2$ otherwise so that
	\[ a = e_1 \wedge e^{s} - e_{s} \wedge e^1,  b = e_1 \wedge e_{s} - e^1 \wedge e^{s}  \in W_0, \]
i.e. such that $S_a, S_b \in \mathfrak{u}(2k,2l)$, a straightforward computation shows that $[S_a, S_b](e_1) = [T^1_a, T^1_b](e_1) = - [T^1_{ia}, T^1_{ib}](e_1) = \beta e^1$ for some $\beta \neq 0$, in particular what we have is
	\[ [V_s, V_s] \neq 0, \qquad \forall \ s = 0,1,2, \]
where $[\cdot , \cdot]$ is the Lie bracket in $\mathfrak{so}(4k,4l)$. To finalize the proof of the Lemma, observe that the Lie bracket induces a $\mathfrak{sp}(k,l)$-homomorphism
	\[ L : \Lambda^2 V_s \rightarrow \mathfrak{so}(4k,4l), \qquad L(a \wedge b) = [a,b], \]
that complexifies to the homomorphism
	\[
		L_\mathbb{C} : \Lambda^2 V(\omega_2) \rightarrow \mathfrak{sp}(2n,\mathbb{C}) \oplus \mathbb{C}^3 \oplus V(\omega_2) \oplus V(\omega_2) \oplus V(\omega_2).
	\]
We have a decomposition as $\mathfrak{sp}(2n,\mathbb{C})$-modules given by Lemma \ref{L1.1}
	\[
		\Lambda^2 V(\omega_2) \cong \mathfrak{sp}(2n,\mathbb{C}) \oplus V(\omega_1 + \omega_3),
	\]
so that Lemma \ref{Schur} and the fact that $L_\mathbb{C} \neq 0$ implies that $L_\mathbb{C}$ is just the projection into the factor $\mathfrak{sp}(2n,\mathbb{C})$, in particular
	\[ [V_s, V_s] = \mathfrak{sp}(k,l), \qquad \forall s = 0,1,2 \]
and the Lemma follows.
\end{proof}

Proposition \ref{real-forms} tells us that $V(\omega_2)$ and $V(2 \omega_1) \cong \mathfrak{sp}(2n,\mathbb{C})$ admit real forms, on the other hand $V(\omega_1)_\mathbb{R}$ and $V(\omega_3)_\mathbb{R}$ are irreducible real representations of $\mathfrak{sp}(k,l)$ (see \cite[Table 5]{Oni}), so we may find the non-trivial irreducible representations of dimension less that $\mathfrak{sp}(k,l)$ as given by the following 

\begin{cor}\label{C1.3}
If $V$ is an irreducible real representation of $\mathfrak{sp}(k,l)$, such that 
	\[ dim(V) \leq dim(\mathfrak{sp}(k,l)), \]
then either $V \cong \mathbb{C}^{2k,2l}_\mathbb{R}$, $V \otimes \mathbb{C} \cong V(\omega_2)$ or $V \otimes \mathbb{C} \cong \mathfrak{sp}(2n,\mathbb{C})$, the last one corresponding to the highest weight $2 \omega_1$.
\end{cor}

\begin{proof}
If $V$ is a real form of a complex irreducible representation $W$, then 
	\[ dim_\mathbb{C}(W) \leq n(2n + 1) \]
for $n = k + l$ so that by Corollary \ref{C1.2} and the previous discusion we have that $W$ is isomorphic to either $V(\omega_2)$ or $V(2 \omega_1)$. If on the other hand $V = W_\mathbb{R}$, where $W$ is an irreducible complex representation, then $dim_\mathbb{R}(V) = 2 dim_\mathbb{C}(W) \leq n(2n + 1)$ so according to Corollary \ref{C1.2}, $W$ is isomorphic to either $V(\omega_1)$ or $V(\omega_3)$ in the case $n = 3$, but the latter is impossible because 
	\[ 2 dim_\mathbb{C}(V(\omega_3)) = 28 > n(2n + 1) = 21, \]
so that $W \cong V(\omega_1)$ and the result follows.
\end{proof}

If $\rho : \mathfrak{sp}(k,l) \rightarrow \mathfrak{gl}(V)$ is a real representation, we say that the inner product induced by a symmetric bilinear form $B$ is \textbf{invariant} under $\mathfrak{sp}(k,l)$ if $\rho(\mathfrak{sp}(k,l)) \subset \mathfrak{so}(V,B)$, so that by our definition of the quaternionic symplectic algebra, $(x,y) = \textrm{Re} \langle x , y \rangle $ is invariant under $\mathfrak{sp}(k,l)$.

\begin{lema}\label{R1.4}
If $V$ is an irreducible real $\mathfrak{sp}(k,l)$-module and $B_1$, $B_2$ are two symmetric bilinear forms invariant under the $\mathfrak{sp}(k,l)$ action, then there is a $\lambda \in \mathbb{R}$ such that $B_2 = \lambda B_1$ and $\mathfrak{so}(V, B_1) = \mathfrak{so}(V,B_2)$. Moreover, if $V$ and $W$ are non-isomorphic irreducible $\mathfrak{sp}(k,l)$-modules and $B$ is an invariant symmetric bilinear form in $V \oplus W$, then 
	\[ B(x,y) = 0, \qquad \forall x \in V, \ y \in W. \]
\end{lema}

\begin{proof}
If both $B_1$ and $B_2$ are zero, the conclusion is trivial so we suppose that $B_1 \neq 0$. Observe that the Kernel of $B_1$
	\[ Ker(B_1) = \{ x \in V : B_1(x,y) = 0, \quad \forall y \in V \} \]
is a $\mathfrak{sp}(k,l)$-submodule of $V$ because $B_1$ is invariant and cannot be $V$ because $B_1$ is not zero, so that $Ker(B_1) = 0$ and $B_1$ is non-degenerated, this implies that there exist a linear map $L \in \mathfrak{gl}(V)$ such that
	\[	B_2(x,y) = B_1(Lx,y), \qquad \forall x,y \in V \]
and the invariance of $B_2$ implies that $L$ is a $\mathfrak{sp}(k,l)$-homomorphism. If we denote $Hom(V)$ the vector space of $\mathfrak{sp}(k,l)$-homomorphisms of $V$, then we can see that it is a finite dimension associative algebra over $\mathbb{R}$ under composition and by Lemma \ref{Schur}, it is a division algebra, i.e. every non-zero element is invertible. Thus, by Frobenius' Theorem (see \cite[Pg. 158]{Fro}), $Hom(V)$ is isomorphic as an algebra to $\mathbb{R}$, $\mathbb{C}$ or the quaternions $\mathbb{H}$. Now $B_1$ and $B_2$ are symmetric so that 
	\[ B_1(Lx,y) = B_1(x , Ly ), \qquad \forall x,y \in V, \]
and thus $L$ is self-adjoint, but the adjoint is precisely the conjugation in the corresponding algebras so that $L = \lambda Id_V$ for some $\lambda \in \mathbb{R}$. Observe that $A \in \mathfrak{so}(V,B_s)$ if by definition
	\[ B_s(Ax , y) + B_s(x,Ay) = 0, \qquad \forall \ x,y \in V \]
and this condition is invariant under multiplication by scalars, so we have $\mathfrak{so}(V, B_1) = \mathfrak{so}(V,B_2)$. Finally, if $B$ is an invariant symmetric bilinear form in $V \oplus W$ and suppose it is not identically zero, because in this case the result follows, this implies that it is non-degenerated when restricted to at least one irreducible submodule, suppose then that $B$ is non-degenerated when restricted to $W$ and consider the linear map
	\[ \widehat{B} : V \rightarrow W^*, \qquad \widetilde{B}(x)(y) = B(x,y), \]
then $\widehat{B}$ is a $\mathfrak{sp}(k,l)$-homomorphism that is either zero or an isomorphism, but 
	\[ W \rightarrow W^*, \qquad x \rightarrow B(x, \cdot) \]
is an isomorphism of $\mathfrak{sp}(k,l)$-modules and $V$ is not isomorphic to $W$, so that we have $\widehat{B} \equiv 0$ and the result follows.
\end{proof}

\begin{cor}\label{C1.4}
If $V$ is a non-trivial real $\mathfrak{sp}(k,l)$-module with $dim_\mathbb{R}(V) \leq 4(k+l)$ and a $\mathfrak{sp}(k,l)$-invariant inner product given by a bilinear form $B$, then $V$ is irreducible with highest weight $\omega_1$ and under the identification $V \cong \mathbb{R}^{4(k+l)}$ as $\mathfrak{sp}(k,l)$-modules $B = \lambda (\cdot , \cdot )$ for some non-zero $\lambda \in \mathbb{R}$. Moreover there is a $\mathfrak{sp}(k,l)$-module decomposition
	\begin{displaymath}
		\mathfrak{so}(V,B) = \mathfrak{sp}(k,l) \oplus \mathfrak{sp}(1) \oplus V_0 \oplus V_1 \oplus V_2
	\end{displaymath}
where $\mathfrak{sp}(1)$ is also a subalgebra of $\mathfrak{so}(4k,4l)$ that commutes with $\mathfrak{sp}(k,l)$, for every $s \in \{0,1,2\}$, $V_s$ is an irreducible $\mathfrak{sp}(k,l)$-submodule whose complexification is isomorphic to $V(\omega_2)$ and 
	\begin{displaymath}
		[V_s , V_s] = [\mathfrak{sp}(k,l) , \mathfrak{sp}(k,l)] = \mathfrak{sp}(k,l),
	\end{displaymath}
where $[\cdot , \cdot]$ is the Lie bracket in $\mathfrak{so}(V,B)$.
\end{cor}

\begin{proof}
By Corollary \ref{C1.3}, $V$ is irreducible with highest weight either $\omega_2$ or $\omega_1$, but 
	\[ \textrm{dim}(V(\omega_2)) > \textrm{dim}(V(\omega_1)), \qquad k+l > 2, \]
so we have that $V$ is isomorphic to the irreducible module $\mathbb{C}^{2k,2l}_\mathbb{R}$ corresponding to the highest weight $\omega_1$, now the result follows from Lemma \ref{R1.4} and Lemma \ref{L1.6}.
\end{proof}

\subsection{Deformations of symmetric pairs}\label{Def}

Recall that for every $k,l \in \mathbb{N}$, the linear structure of the hermitian space $\mathbb{H}^{k,l}$ is given by
	\[ Sp(k,l) \times Sp(1) \times \mathbb{H}^{k,l} \rightarrow \mathbb{H}^{k,l}, \qquad
		(g, \eta, Z) \mapsto g Z \eta^{-1},
		\]
and the infinitesimal version of this is 
	\begin{equation}\label{action}	\mathfrak{sp}(k,l) \times \mathfrak{sp}(1) \times \mathbb{H}^{k,l} \rightarrow \mathbb{H}^{k,l}, \qquad 
		(X, \zeta, Z) \mapsto X Z - Z \zeta.
	\end{equation}
This linear information is encoded in the Lie bracket structure of the symplectic algebras in the next dimension as given by the following proposition, in this section we study uniqueness properties of this information.

\begin{prop}\label{PA.1}
There is a decomposition of stable $\mathfrak{sp}(k,l) \oplus \mathfrak{sp}(1)$-modules
	\[ \mathfrak{sp}(k,l+1) = \mathfrak{sp}(k,l) \oplus \mathfrak{sp}(1) \oplus \mathbb{H}^{k,l} \]
when considered with the commutator structure $[X,Y] = XY - YX$. Moreover there is a $\mathfrak{sp}(k,l) \oplus \mathfrak{sp}(1)$-homomorphism induced by the Lie bracket
	\[ \Omega : \Lambda^2 \mathbb{H}^{k,l} \rightarrow \mathfrak{sp}(k,l) \oplus \mathfrak{sp}(1), \qquad \Omega(x \wedge y) = [x,y], \]
given explicitly by
	\[ \Omega (Z \wedge W ) = \left(W Z_ 0^* - Z W_0^*\right) + \left(W_0^* Z - Z_0^* W \right), \]
where $Z_0 = - I_{k,l} Z$, and $W^* = \overline{W}^t$. Analogously, there is a decomposition of $\mathfrak{sp}(k,l) \oplus \mathfrak{sp}(1)$-modules
	\[ \mathfrak{sp}(k+1,l) = \mathfrak{sp}(k,l) \oplus \mathfrak{sp}(1) \oplus \mathbb{H}^{k,l} \]
together with a $\mathfrak{sp}(k,l) \oplus \mathfrak{sp}(1)$-homomorphism induced by the Lie bracket of $\mathfrak{sp}(k+1,l)$
	\[ \Omega' : \Lambda^2 \mathbb{H}^{k,l} \rightarrow \mathfrak{sp}(k,l) \oplus \mathfrak{sp}(1),  \]
that we can see is just $\Omega' = - \Omega$.
\end{prop}

\begin{proof}
Consider the inclusions
	\begin{equation}\label{inclusion1} \mathfrak{sp}(k,l) \times \mathfrak{sp}(1) \times \mathbb{H}^{k,l} \hookrightarrow \mathfrak{sp}(k,l+1), \qquad
		(X, \zeta, Z) \mapsto \left(\begin{array}{cc} A & -Z \\ Z_0^* & \zeta \end{array}\right)
		\end{equation}
and
	\begin{equation}\label{inclusion2} \mathfrak{sp}(k,l) \times \mathfrak{sp}(1) \times \mathbb{H}^{k,l} \hookrightarrow \mathfrak{sp}(k+1,l), \qquad 
		(X, \zeta, Z) \mapsto \left(\begin{array}{cc} \zeta & Z_0^* \\ Z & A \end{array}\right)
		\end{equation}
then in both cases, the $\mathfrak{sp}(k,l) \oplus \mathfrak{sp}(1)$-structure induced by the Lie bracket is precisely the structure defined in (\ref{action}).
\end{proof}

\begin{remark}\label{RA.1}
Recall that $(\mathfrak{s}, \mathfrak{t})$ is called a symmetric pair if $\mathfrak{s}$ is a Lie algebra and $\mathfrak{t}$ is the subalgebra of fixed points under an involutive automorphism of $\mathfrak{s}$, equivalently there is a $\mathfrak{t}$-stable decomposition $\mathfrak{s} = \mathfrak{t} \oplus W$ such that $[W,W] \subset \mathfrak{t}$. The previous two injections gives us symmetric pairs $(\h , \g \oplus \mathfrak{k})$, where $\g = \mathfrak{sp}(k,l)$, $\mathfrak{k} = \mathfrak{sp}(1)$ and $\h$ is either $\mathfrak{sp}(k,l+1)$ or $\mathfrak{sp}(k+1,l)$. Observe that in general, given the decomposition $\mathfrak{s} = \mathfrak{t} \oplus W$, the Lie algebra structrue of $\mathfrak{s}$ is completely determined by its $\mathfrak{t}$-module structure together with a $\mathfrak{t}$-homomorphism
	\[
		\Psi : \Lambda^2 W  \rightarrow \mathfrak{t},
			\]
but not every $\mathfrak{t}$-homomorphism gives a Lie algebra structure.
\end{remark}

Let $\h$ be either $\mathfrak{sp}(k,l+1)$ or $\mathfrak{sp}(k+1,l)$ and consider the decomposition
	\[ \h = \g \oplus \mathfrak{k} \oplus V, \]
where  $\g \cong \mathfrak{sp}(k,l)$, $\mathfrak{k} = \mathfrak{sp}(1)$ and $V = \mathbb{H}^{k,l}$. Denote by $\mathfrak{m} = \g \oplus V$ and consider the skew-symmetric bilinear map $[ \cdot , \cdot ]_m : \mathfrak{m} \times \mathfrak{m} \rightarrow \mathfrak{m}$ induced by the Lie bracket of $\h$ followed by the orthogonal projection $\h = \mathfrak{k} \oplus \mathfrak{m} \rightarrow \mathfrak{m}$.

\begin{prop}\label{PA.2}
There exist three elements $X,Y,Z \in V$ such that 
	\[
		[[Y,Z]_m , X]_m + [ [X , Y]_m, Z ]_m + [ [Z, X ]_m , Y]_m \neq 0.
			\] 
\end{prop}

\begin{proof}
Consider first the case where $(k,l)$ is either $(1,0)$ or $(0,1)$, in these cases, the elements
	\[ 
		X = \left(\begin{array}{cc}	0 & 1 \\ -1 & 0 \end{array}\right), \quad	Y = \left(\begin{array}{cc}	0 & i \\ i & 0 \end{array}\right), \quad
	Z = \left(\begin{array}{cc}	0 & j \\ j & 0 \end{array}\right) \in V,
			\]
give us the relations
	\[ [ Y, [X , Z ]_m ]_m = [ [X , Y]_m, Z ]_m = - [X , [Y,Z]_m ]_m = \pm \left(\begin{array}{cc} 0 & 2 ij \\ 2 ij & 0 \end{array}\right)	\]
and the result follows. The general cases follow from the previous ones by observing that we can consider a block-diagonal embedding $\mathfrak{sp}(1,1) \hookrightarrow \h$ analogous to (\ref{inclusion1}) and (\ref{inclusion2}) where the corresponding Lie brackets preserve the decomposition $\mathfrak{sp}(1) \oplus \mathfrak{sp}(1) \oplus \mathbb{H} \hookrightarrow \g \oplus \mathfrak{k} \oplus V$.
\end{proof}

\begin{cor}\label{CA.1}
The non-associative algebra $(\mathfrak{m}, [\cdot , \cdot ]_m)$ is not a Lie algebra.
\end{cor}

Observe that the $\g \oplus \mathfrak{k}$-homomorphism induced by the Lie bracket of $\h$
	\[ \Omega : \Lambda^2 V \rightarrow \g \oplus \mathfrak{k} \]
splits in two components $\Omega = \Omega_\g + \Omega_\mathfrak{k}$ when projected to the corresponding factor and both of these linear maps are homomorphisms when considered with the corresponding module structure, induced by $\g$ or $\mathfrak{k}$. For every $r,s \in \mathbb{R}$, denote $\Omega_{r,s} = r\ \Omega_\g + s\ \Omega_\mathfrak{k}$, and so, we can define the non-associative algebra $\h_{r,s}$ given by the $\g \oplus \mathfrak{k}$-module structure $\h_{r,s} = \g \oplus \mathfrak{k} \oplus V$, together with the linear map $\Omega_{r,s}$, i.e.
	\[ [x,y] := \Omega_{r,s}(x \wedge y), \qquad \forall \ x,y \in V \subset \h_{r,s}. \]

\begin{lema}\label{LA.1}
The non-associative algebra $\h_{r,s}$ is a Lie algebra if and only if $r = s$, moreover if $r > 0$, then $\h_{r,r}$ is a Lie algebra isomorphic to $\h$. If $\h \cong \mathfrak{sp}(k+1,l)$ then $\h_{-1,-1} \cong \mathfrak{sp}(k,l+1)$ and conversely.

\end{lema}

\begin{proof}
Take $X,Y,Z \in V$, we use the following notation
	\[
		\Omega_{r,s}(X,Y,Z) := \Omega_{r,s}(X,Y)Z + \Omega_{r,s}(Z,X)Y + \Omega_{r,s}(Y,Z)X,
			\]
so that the only condition left for $\h_{r,s}$ to be a Lie algebra is that $\Omega_{r,s}(X,Y,Z) = 0$ for every $X,Y,Z \in V$. Observe that $\Omega_\g = \Omega_{1,0}$, $\Omega_\mathfrak{k} = \Omega_{0,1}$ and $\h_{1,1}$ is the Lie algebra $\h$, so by Jacobi identity in $\h$ we have that
	\[
		\Omega_\g(X,Y,Z) = - \Omega_\mathfrak{k}(X,Y,Z), \qquad \forall \ X,Y,Z \in V;
			\]
thus $\Omega_{r,s}(X,Y,Z) = (r-s) \Omega_\g(X,Y,Z)$. Observe that $\Omega_\g(x,y) = [x,y]_\m$ so that by Proposition \ref{PA.2}, there exist $X,Y,Z \in V$ such that $\Omega_\g(X,Y,Z) \neq 0$, so $\Omega_{r,s}$ is identically zero only when $r = s$. For the second part consider the linear map
	\[
		\Phi : \g \oplus \mathfrak{k} \oplus V \rightarrow \g \oplus \mathfrak{k} \oplus V
			\]
defined by $\Phi(X + v) = X + \sqrt{r} v$, for $X \in \g \oplus \mathfrak{k}$ and $v \in V$, then $\Phi$ is an isomorphism of Lie algebras between $\h_{r,r}$ and $\h$ and the result follows from the last part of Proposition \ref{PA.1}.
\end{proof}

\begin{lema}\label{LA.2}
Suppose $(\mathcal{H}, \g \oplus \mathfrak{k})$ is a symmetric pair with $\mathcal{H}$ a simple Lie algebra and the decomposition of $\g \oplus \mathfrak{k}$-modules is given by
	\[ \mathcal{H} = \g \oplus \mathfrak{k} \oplus V, \]
where $V \cong \mathbb{H}^{k,l}$ as a $\g \oplus \mathfrak{k}$-module, then $\mathcal{H}$ is isomorphic to $\h$ either $\mathfrak{sp}(k+1,l)$ or $\mathfrak{sp}(k,l+1)$, and such isomorphism is an isomorphism of symmetric pairs $(\mathcal{H}, \g \oplus \mathfrak{k}) \cong (\h, \g \oplus \mathfrak{k})$, where the latter is given by the inclusions (\ref{inclusion1}) and (\ref{inclusion2}).
\end{lema}

\begin{proof}
Consider the isomorphism of $\g \oplus \mathfrak{k}$-modules
	\[	\mathcal{H} = \g \oplus \mathfrak{k} \oplus V \rightarrow \mathfrak{sp}(k,l+1) = \g \oplus \mathfrak{k} \oplus V, \]
then the Lie bracket of $\mathcal{H}$ defines a $\g \oplus \mathfrak{k}$-homomorphism $\Xi : \Lambda^2 V \rightarrow \g \oplus \mathfrak{k}$, and by Corollary \ref{C1.4} we have a decomposition of $\g \oplus \mathfrak{k}$-modules 
	\[ \Lambda^2 V \cong \g \oplus \mathfrak{k} \oplus V_0 \oplus V_1 \oplus V_2, \]
such that $V_j$ is not isomorphic to $\g$ nor $\mathfrak{k}$, so that by Lemma \ref{Schur} we have that $\Xi = \Omega_{r,s}$ for some $r,s \in \mathbb{R}$, but by the previous Lemma, $r = s$ and so $\mathcal{H}$ is isomorphic to either $\mathfrak{sp}(k,l+1)$ or $\mathfrak{sp}(k,l+1)$, where we achived such isomorphism preserving the $\g \oplus \mathfrak{k}$-decomposition.
\end{proof}

\begin{teo}\label{TA.1}
If $G \times K$ is a connected subgroup of $Sp(p,q)$, where $G$ is isomorphic to either $Sp(p-1,q)$ or $Sp(p,q-1)$ and $K \cong Sp(1)$, then there exists 
	\[ \phi : Sp(p,q) \rightarrow Sp(p,q) \]
an automorphism such that $\phi(G \times K) \subset Sp(p,q)$ is the block-diagonal embedding as in (\ref{inclusion1}) and (\ref{inclusion2}).
\end{teo}

\begin{proof}
Let $\g \times \mathfrak{k} \subset \h$ be the Lie algebras corresponding to $G \times K \subset H$, with $H = Sp(p,q)$. As $\g \times \mathfrak{k}$ is a semisimple Lie algebra, then complete reducibility of real representations implies that there is a subspace $\mathfrak{p} \subset \h$ that is $\g \times \mathfrak{k}$-invariant under the adjoint action and $\h = \g \oplus \mathfrak{k} \oplus \mathfrak{p}$. Observe that $\mathfrak{l} = \mathfrak{k} \oplus \mathfrak{p}$ is a $\g$-submodule so that if it is a trivial module, then
	\[ [\mathfrak{l} , \mathfrak{l} ] \subset \mathfrak{l}, \]
and then $\mathfrak{l}$ is a non-trivial ideal of $\h$, but this is impossible since $\h$ is a simple Lie algebra. Using the restriction of dimensions and Corollary \ref{C1.3}, this implies that $\mathfrak{p} \cong \mathbb{H}^{k,l}$ as $\g$-modules, in particular it is irreducible. Now $\mathfrak{k} \subset Z_\h(\g)$ so that 
	\[ ad(\mathfrak{k}) \subset Hom_\g (\mathfrak{p}) \cap \mathfrak{so}(\mathfrak{p}, B) \cong \mathfrak{sp}(1), \]
where $B$ is the Killing form of $\h$. As before, if $\mathfrak{p}$ is a trivial $\mathfrak{k}$-module, then $\mathfrak{k}$ is an ideal of $\h$ but this is again impossible, so that we have
	\[ ad(\mathfrak{k}) = Hom_\g (\mathfrak{p}) \cap \mathfrak{so}(\mathfrak{p}, B) \cong \mathfrak{sp}(1) \]
and thus $\mathfrak{p} \cong \mathbb{H}^{k,l}$ as $\g \times \mathfrak{k}$-modules and $(\h, \g \times \mathfrak{k})$ is a symmetric pair. The result then follows from Lemma \ref{LA.2}. 
\end{proof}

\section{The homogeneous manifolds $Sp(1) \bs Sp(p,q)$}\label{Stiefel}

In all this section we fix the notation so that $H = Sp(p,q)$, $K = Sp(1)$ and $G$ is either $Sp(p-1,q)$ or $Sp(p,q-1)$; denote also by $\h$, $\mathfrak{k}$ and $\g$ their corresponding Lie algebras. Observe that the bilinear form given by $B(x,y) = \textrm{Re tr}(XY)$ is non-degenerated in $\h$ because is a multiple of the Killing form, moreover there is a decomposition $\h = \g \oplus \mathfrak{k} \oplus \mathfrak{p}$ orthogonal with respect to $B$ so that $\mathfrak{p} \cong \mathbb{C}^{2k,2l}$ as a $\g \oplus \mathfrak{k}$ with respect to the adjoint representation.

\subsection{Proof of Theorem \ref{C3.3}}
It is enough to prove the Theorem for $\varphi_0 : G \times K \rightarrow H$ to be a block-diagonal homomorphism as in (\ref{inclusion1}) and (\ref{inclusion2}) respectively. Theorem \ref{TA.1} tells us that there exists an automorphism of $H$ that sends $\varphi(G \times K)$ into $\varphi_0(G \times K)$, so that it induces the desired analytic diffeomorphism. Now the fact that we have $R(H) \subset Iso(\varphi(K) \bs H, \overline{h})$ implies that $\overline{h}$ is a right invariant pseudo-Riemannian metric that is thus induced by a bilinear form $D$ in $\g \oplus \mathfrak{p}$ that is $Ad(K)$-invariant, and the fact that we also have $L(G) \subset Iso(\varphi(K) \bs H, \overline{h})$ implies that $D$ is also $Ad(G)$-invariant that splits as $D = D_1 + D_2$, where $D_1$ and $D_2$ are the restrictions to $\g$ and $\mathfrak{p}$ respectively. But $\g$ and $\mathfrak{p}$ are non-isomorphic irreducible $Ad(G)$-modules, so that Lemma \ref{R1.4} implies that $D_1 = a B_1$ and $D_2 = b B_2$ for $a,b \in \mathbb{R}$ non-zero where $B_1$ and $B_2$ are the restrictions of $B$ to $\g$ and $\mathfrak{p}$ respectively. This tells us that $\overline{h}$ is a rescaling of $\overline{g}$ over the $G$-orbits and its orthogonal complements and the result follows.

\subsection{Affine structure and isometries of $Sp(1) \backslash Sp(p,q)$}

Theorem \ref{C3.3} tells us that we may obviate the embedding $G \times K \hookrightarrow H$ and thus consider just the homogeneous space $K \backslash H$ with the left $G$-action given by left multiplications. Fix the pseudo-Riemannian metrics in $H$ and $K \bs H$ induced from $B$, recall that an element $X \in \h$ generates two distinct killing fields in $H$ induced by $B$ via left and right multiplication of the group
	\[	X^+_{h} = \frac{d}{dt}_{|_{t=0}} h \cdot exp(tX), \qquad X^*_{h} = \frac{d}{dt}_{|_{t=0}} exp(tX) \cdot h, \qquad \forall \ h \in H.	\]
In the homogeneous space $K \bs H$ only one of this killing fields is well defined, namely the one defined by right multiplications $X^+_{K h} = \frac{d}{dt}_{|_{t=0}} K h \cdot exp(tX)$.

\begin{lema}\label{L3.1}
If $\pi : H \rightarrow K \bs H$ denotes the natural projection, the linear map $d \pi_e$ gives the identification $\m \cong T_{eK} ( K \bs H )$ and under this identification 
	\[ 2 \left( \nabla_{X^+} Y^+ \right)_{Ke} = [X,Y]_\m, \qquad  \forall \ X,Y \in \m, \]
 where $\nabla$ is the Levi-Civita connection in $K \bs H$ induced by the pseudo-Riemannian submersion and $[X,Y]_\m$ is the orthogonal projection of $[X,Y]$ to $\m$.
\end{lema}

\begin{proof}
Take $X \in \m$ and $\{X_j\}$ a B-orthonormal basis of $\mathfrak{k}$, then the vector field defined by $\overline{X} = X^+ - \sum_j \overline{g}( X^+ , X_j^* ) X_j^*$ is orthogonal to the $K$-orbits in $H$ and its projection is precisely $X^+$ in $K \bs H$. Moreover, if $D$ denotes the Levi-Civita connection in $H$ and $Y \in \m$, then 
	\[ D_{\overline{X}} \overline{Y} = D_{X^+} Y^+ + \sum_j \left( \overline{g}( X^+ , X_j^* ) Z_j^1 + \overline{g}( Y^+ , X_j^* ) Z_j^2 + a_j X_j^* \right),  \]
for some smooth vector fields $Z_j^s$ and functions $a_j$. Observe that 
	\[ \langle Z^+ , X_j^* \rangle_{e} = B(Z,X_j) = 0, \qquad \forall \ Z \in \m \]
and $2 \left( D_{X^+} Y^+ \right)_e = [X,Y]$ because the metric is bi-invariant \cite[Corollary 11.10]{O}, so we get
	\[ 2 \left(D_{\overline{X}} \overline{Y} \right)_e = [X,Y]_\m + W\]
for some $W \in \mathfrak{k}$, the result now follows from O'Neill's formula for Levi-Civita connections under pseudo-Riemannian submersions \cite[Lemma 7.45]{O}.
\end{proof}

The bilinear form $[\cdot , \cdot ]_\m$ induces in $\m$ the structure of a non-associative algebra that has as its automorphism group
	\[ Aut(\m) = \{ T \in GL(\m) : T[X, Y]_\m = [TX, TY]_\m, \quad \forall \ X,Y \in \m \}, \]
that is an algebraic Lie group with Lie algebra
	\[	Der(\m) = \{ T \in GL(\m) : T[X, Y]_\m = [TX, Y]_\m + [X, TY]_\m, \ \forall \ X,Y \in \m \}. \]
Consider the isometry group of $K \bs H$ with the given pseudo-Riemannian structure and denote it simply by $Iso(K \bs H)$, consider also the isotropy subgroup of elements that fix the identity class
	\[ Iso(K \bs H, Ke) = \{ \varphi \in Iso(K \bs H) : \varphi(Ke) = Ke \},	\]
then we denote the isotropy representation as
	\[ \begin{array}{rcl}
		\lambda_e : Iso(K \bs H, Ke) & \rightarrow & GL(\m) \\
		\varphi & \mapsto & d\varphi_e.
	\end{array}\]

\begin{cor}\label{C3.1}
$\lambda_e$ is injective with closed image contained in $Aut(\m)$.
\end{cor}

\begin{proof}
The first part follows from the fact that an isometry of a connected pseudo-Riemannian manifold is completely determined by its value and derivative in a single point, see \cite[Sec. I]{Ko}. The second part follows from Lemma \ref{L3.1} and the fact that $\psi_*(\nabla_U V ) = \nabla_{\psi_*(U)} \psi_*(V)$ for every isometry $\psi \in Iso(K \bs H)$ and $U,V$ local vector fields in $K \bs H$.
\end{proof}

Denote by $\widetilde{ad}_x(y) = [x,y]_\m$ the adjoint representation in $\m$ so that if $\m$ is invariant under $ad_x$, then $\widetilde{ad}_x(y) = ad_x(y)$ and thus $\widetilde{ad}_x \in Der(\m)$, in particular we obtain that $\widetilde{ad}(\g \oplus \mathfrak{k}) \subset Der(\m)$.

\begin{lema}\label{L3.2}
If $x \in \mathfrak{p}$ is such that $\widetilde{ad}_x \in Der(\m)$, then $x = 0$.
\end{lema}

\begin{proof}
Recall that $Der(\m)$ is a Lie subalgebra of $\mathfrak{gl}(\m)$, because the derivations of any algebra is a Lie algebra and a straightforward computation shows that 
	\begin{displaymath}
		[\delta , \widetilde{ad}_{x}] =  \widetilde{ad}_{(\delta x)}, \qquad \forall x \in \m, \ \delta \in Der(\m).
	\end{displaymath}
If we define define $\mathfrak{p}_0 = \{ x \in \mathfrak{p} : \widetilde{ad}_x \in Der(L) \}$, then $\mathfrak{p}_0$ is a $\g$-submodule of the irreducible $\g$-module $\mathfrak{p}$, so that $\mathfrak{p}_0$ is either $0$ or $\mathfrak{p}$. Now Proposition \ref{PA.2} shows that there exist three elements $X,Y,Z \in \mathfrak{p}$ such that 
	\[ \widetilde{ad}_Z [X,Y]_\mathfrak{m} \neq [\widetilde{ad}_Z X,Y]_\mathfrak{m} + [X, \widetilde{ad}_Z Y ]_\mathfrak{m}, \]
in particular $Z \notin \mathfrak{p}_0$ and thus $\mathfrak{p}_0 = 0$.
\end{proof}

\begin{prop}\label{P3.3}
$Der(\m) = \widetilde{ad}(\g \oplus \mathfrak{k})$
\end{prop}

\begin{proof}
Take $\delta \in D = Der(\m)$ and $x \in \g$, then we have a decomposition $\delta x = x_1 + x_2$, where $x_1 \in \g$, $x_2 \in \mathfrak{p}$ and
	\begin{displaymath}
		[\delta , ad_x] = \widetilde{ad}_{(\delta x)} = ad_{x_1} + \widetilde{ad}_{x_2} \in D,
	\end{displaymath}
so $[\delta , ad_x] - \widetilde{ad}_{x_1} = \widetilde{ad}_{x_2} \in D$ and by Lemma \ref{L3.2}, $x_2 = 0$, this implies that $I := \widetilde{ad}(\g)$ is a simple ideal of the Lie algebra $D$, so the killing form of $I$ is non-degenerated and is precisely the restriction of $k_D$ (the killing form of $D$).

Take $I_0 = I^\perp$ the orthogonal complement of $I$ in $D$ with respect to $k_D$ so that we have a direct sum decomposition $D = I \oplus I_0$, where $I_0$ is again an ideal of $D$ (this is a consequence of the fact that $k_D$ is invariant under the Lie bracket of $D$). Moreover we have $[I_0 , I ] = I_0 \cap I = \{0\}$, this implies that for every $x \in \g$ and $\delta \in I_0$, $0 = [\delta, ad_x] = ad_{(\delta x)}$, so $\delta( \g) = 0$ because $\widetilde{ad}$ is injective when restricted to $\g$. Thus, for every $\delta \in I_0$ and $x \in \g$, we have $\delta \circ ad_x = ad_x \circ \delta$ and $\delta( \g) = 0$, so the restricted map $\delta_{|_\mathfrak{p}} : \mathfrak{p} \rightarrow \mathfrak{p}$ is a $\g$-homomorphism. Recall that $Hom_\g(\mathfrak{p}) \cong \mathbb{H}$ as an associative algebra under the composition (see the proof of Lemma \ref{R1.4}), and as $\mathfrak{k}$ commutes with $\g$, we have that 
	\begin{displaymath}
		\mathfrak{sp}(1) \cong \widetilde{ad}(\mathfrak{k}) \subset I_0 \subset Hom_\g(\mathfrak{p}) \cong \mathbb{H},
	\end{displaymath}
so $I_0 = \widetilde{ad}(k)$ and the proposition follows.
\end{proof} 

If $\m$ is $Ad_h$-invariant for some $h \in H$, then we denote the restricted map by $\widetilde{Ad}_h$ so that we have $\widetilde{Ad}(G \times K) \subset Aut(\m)$.

\begin{cor}\label{C3.2}
$Aut(\m)$ is an algebraic group with $\widetilde{Ad}(G \times K)$ as a finite index subgroup.
\end{cor}

\begin{proof}
$Aut(\m)$ is an algebraic group because it is the automorphism group of a bilinear map and the subgroup $\widetilde{Ad}(G \times K)$ is a Lie subgroup with Lie algebra $\widetilde{ad}(\g \oplus \mathfrak{k}) = Der(\m) = Lie(Aut(\m))$, this implies that $\widetilde{Ad}(G \times K)$ is the connected component of the identity of $Aut(\m)$, the result now follows from the fact that an algebraic group has only finitely many connected components \cite[Thm. 3.6]{P-R}.
\end{proof}

\subsection{Proof of Theorem \ref{T3.1}}

Consider the conjugation map
	\[ C : G \times K \rightarrow Iso(K \bs H), \qquad C(g) = L_g \circ R_{g^{-1}}, \]
so that $C(G \times K) \subset Iso(K \bs H, Ke)$ and $\lambda_e \circ C(G \times K) = \widetilde{Ad}(G \times K)$. Corollary \ref{C3.1}, Corollary \ref{C3.2} and the fact that 
	\[	\widetilde{Ad}(G \times K) \subset \lambda_e(Iso(K \bs H, Ke)) \subset Aut(\m)	\]
imply that $\lambda_e(Iso(K \bs H, Ke))$ is a Lie subgroup of $Aut(\m)$ consisting of a finite number of connected components, in particular, $C(G \times K)$ is the connected component of the identity and has finite index in $Iso(K \bs H, Ke)$. The group $Iso( K \bs H)$ acts transitively on the connected manifold $K \bs H$ and the isotropy subgroup $Iso(K \bs H, Ke)$ has finitely many components, this implies that $Iso(K \bs H)$ has finitely many components, see for example \cite[Lemma 2.1]{Q1}. 
Denote by $L$ the connected component of the identity of $Iso(K \bs H)$ so that we have
		\begin{equation}\label{dim-isot} dim(L \cap Iso(K \bs H, Ke) ) = dim(C(G \times K)),	\end{equation}
now the action of $L$ in $K \bs H$ is also transitive, so we have a double identification
	\[	L /(L \cap Iso(K \bs H, K e)) \cong K \bs H \cong L(G)R(H) / C(G \times K),	\]
these identifications together with identity (\ref{dim-isot}) tells us that we have an inclusion of connected groups of the same dimension $L(G)R(H) \subset L$, then they must be equal. Finally, the homomorphism
	\begin{displaymath}
		L \times R^{-1} : G \times H \rightarrow Iso(K \bs H)_0
	\end{displaymath}
is a covering homomorphism of Lie groups, so its kernel must be contained in the center, that from Proposition \ref{P1.7} consists of the elements $(\pm e_G, \pm e_H)$, but we can see directly that $(e_G, -e_H)$ and $(-e_G, e_H)$ are not contained in such kernel.

\subsection{Proof of Theorem \ref{T3.2}}

Let us denote as before $H = Sp(p,q)$ and $K = Sp(1)$, by hypothesis there exists an homomorphism
	\begin{displaymath}
		\eta : G \rightarrow Iso(K \backslash H)
	\end{displaymath}
that commutes with the right $H$ action, as $G$ is simply connected, Theorem \ref{T3.1} implies that there exist two homomorphisms $\rho_1 : G \rightarrow G$ and $\rho_2 : G \rightarrow H$, such that $\eta(g) = L_{\rho_1(g)} \circ R_{\rho_2(g)^{-1}}$. The commutativity with $H$ is the same as to say that for all $h \in H$ and $g \in G$ we have
	\begin{displaymath}
		L_{\rho_1(g)} \circ R_{\rho_2(g)^{-1}} \circ R_{h^{-1}} = L_{\rho_1(g)} \circ R_{h^{-1}} \circ R_{\rho_2(g)^{-1}},
	\end{displaymath}
this implies that $R_{\left(\rho_2(g) h \right)^{-1}} = R_{\left(h \rho_2(g) \right)^{-1}}$, and so by a straighforward computation, this implies $[\rho_2(g), h] = \rho_2(g) h \rho_2(g)^{-1} h^{-1} \in K$. We have that $G' = \rho_2(G)$ is a simple Lie group of dimension less than or equal to $G$, then $[G',G']$ is a simple Lie group of the same dimension as $G'$ that by the previous analysis is contained in $K$, so the only possibility for this is that $G' = \{e\}$, $\rho_2$ is trivial and then $\eta(g) = L_{\rho(g)}$, for $\rho := \rho_1$. We can write the manifold as
	\begin{displaymath}
		M \cong \left(K \backslash H \right) / \Lambda, \quad \textrm{where} \quad \Lambda = \pi_1(M) \leq Iso(K \backslash H),
	\end{displaymath}
then $\Lambda$ acts in the set of connected components of $Iso(K \backslash H)$ by translations and $\Lambda_0 = \Lambda \cap Iso_0(X, \overline{e})$ is the isotropy of the connected component of the identity $Iso_0(X, \overline{e})$ under this action, Theorem \ref{T3.1} implies that $\Lambda_0 \subset L(G) R(H)$ is a finite index subgroup of $\Lambda$. Consider the homomorphism realizing the connected component of the identity
	\[ (L,R) : G \times H \rightarrow Iso_0(X, \overline{e}), \]
then $\Lambda_1 = (L,R)^{-1}(\Lambda_0)$ is a discrete subgroup of $G \times H$ and the fact that the $G$-action descends to $M$ is equivalent to the condition  
	\begin{displaymath}
		L_{\gamma_1} \circ L_g = L_g \circ L_{\gamma_1}, \quad \forall g \in G, \ (\gamma_1,\gamma_2) \in \Lambda_1.
	\end{displaymath}
As $G \cap K = \{e\}$, this implies that $\gamma_1 \in Z(G) \cong \mathbb{Z}_2$, so $\Gamma = \Lambda_1 \cap \left(\{e\} \times H\right)$ is a finite index subgroup of $\Lambda_1$ that is completely contained in $H$, we thus have that $\widehat{M} := K \backslash H / \Gamma \rightarrow M$ is a finite covering and then $\widehat{M}$ also has finite volume. The projection map $H / \Gamma \rightarrow K \backslash H / \Gamma$ is a smooth fibration with compact fiber so that $H / \Gamma$ also has finite measure with the proyected Haar measure, that is $\Gamma \subset H$ is a lattice.

\section{Isometric actions of $Sp(k,l)$}\label{arbitrary}

In this chapter, $M$ denotes a connected, finite volume, complete, analytic pseudo-Riemannian manifold admiting an analytic isometric action of a symplectic group $G = Sp(k,l)$ such that $k,l \geq 1$, $n = k+l \geq 3$. Moreover we suppose that the $G$-action in $M$ is topologically transitive, i.e. it has a dense orbit.

\subsection{Gromov's centralizer theorem}

Observe that thanks to the fact that $Sp(k,l)$ is simply connected, the action lifts to an analytic isometric action to the universal covering of $M$ that is again complete and we denote it as $\widetilde{M}$. In the present context there is a foliation on $\widetilde{M}$ denoted by $\mathcal{O}$ generated by the $G$-action whose tangent bundle is trivializable via
	\[ \widetilde{M} \times \g \cong T \mathcal{O}, \qquad (p,X) \mapsto X^*_p, \]
where $X^*_p := \frac{d}{dt}_{|_{t=0}} exp(tX) \cdot p$, so that under this trivialization, the pseudo-Riemannian metric corresponds to a rescaling of the Killing form of $\g$. Denote by $Kill(\widetilde{M})$ the Lie algebra of Killing fields in $\widetilde{M}$ and by $Kill_0(\widetilde{M},x)$ the subalgebra of Killing fields that vanish at $x$, then there is a well define homomorphism
	\[ \lambda_x : Kill_0(\widetilde{M}, x) \rightarrow \mathfrak{so}(T_x \widetilde{M}), \]
defined by $\lambda_x(Y)(u) = [Y,U]_x,$ where $U$ is any vector field extending $u$ in a neighborhood of $x$, see \cite{Q2} for a proof of these assertions. In \cite{G}, M. Gromov proved that in the presence of the $Sp(k,l)$-action with a dense orbit, there is a large number of Killing fields centralizing the action. The following version of Gromov's theorem can be found in \cite{Q2} for germs of Killing fields, then such germs extend to global Killing fields in $\widetilde{M}$ because it is analytic and simply connected, see \cite{G}, \cite{Q-C} and \cite{N}.

\begin{prop}\label{P4.1}
There exists an open dense subset $U_0 \subset \widetilde{M}$ such that for all $x \in U_0$, there is an injective homomorphism $\rho_x : \g \rightarrow Kill(\widetilde{M})$ that is an isomorphism onto its image $\g(x) := \rho_x(\g)$. Morever 

	\begin{enumerate}
\item $\g(x) \subset Kill_0(\widetilde{M}, x)$, i.e. every element of $\g(x)$ vanishes at $x$.

\item For every $X,Y \in \g$,
	\begin{displaymath}
		[\rho_x(X),Y^* ] = [X,Y]^* = -[X^*,Y^*]
	\end{displaymath}
in a neighborhood of $x$, thus the elements of $\g(x)$ and their corresponding local flows  preserve both $\mathcal{O}$ and $T \mathcal{O}^\perp$ in a neighborhood of $x$. 

\item The representation $\lambda_x \circ \rho_x : \g \rightarrow \mathfrak{so}(T_x \widetilde{M})$ induces a $\g$-module structure in $T_x \widetilde{M}$ such that $T_x \mathcal{O}$ and $T_x \mathcal{O}^\perp$ are $\g$-submodules, where $\lambda_x$ is the isotropy representation of $Kill_0(\widetilde{M},x)$.
	\end{enumerate}
\end{prop}
Define the \textbf{centralizer of the action} as the Lie algebra 
	\[ \mathcal{H} = \{Y \in Kill(\widetilde{M}) : [Y,X^*] = 0, \quad \forall X \in \g \},	\]
the following Theorem tells us that $\mathcal{H}$ is transitive in an open dense subset, this is a consequence of Gromov's open-dense Theorem and Gromov's centralizer Theorem, see \cite[Lemma 4.1]{Z3} for a sketch of its proof.

\begin{teo}\label{T4.1}
There is an open dense subset $U_1 \subset \widetilde{M}$ such that the evaluation map
	\[ ev_x : \mathcal{H} \rightarrow T_x \widetilde{M}, \qquad ev_x(X) = X_x, \]
is surjective for every $x \in U_1$.
\end{teo}

\begin{remark}\label{R4.1}
For every $x \in U = U_0 \cap U_1$, consider the homomorphism of $\g$ into $\mathcal{H}$ given by
	\[	\widehat{\rho}_x : \g  \rightarrow \mathcal{H}, \qquad  \widehat{\rho}_x = \rho_x(X) + X^*, \]
whose image we denote as $\widehat{\rho}_x (\g) = \mathcal{G}(x)$, this homomorphism induces a $\g$-module structure in $\mathcal{H}$ via the adjoint representation and a $\g$-module structure in $T_x \widetilde{M}$ via the isotropy representation that coincides with $\lambda_x \circ \rho_x$, so that $ev_x$ is a $\g$-homomorphism. Moreover if we take the pull-back of the pseudo-Riemannian metric $g_x$ in $T_x \widetilde{M}$ under the evaluation map, then we get a $\g$-invariant bilinear form in $\mathcal{H}$.
\end{remark}

Suppose from now on that the $G$-orbits are non-degenerated leaves of the foliation $\mathcal{O}$, for example if $dim(M) < 2 dim(G)$ \cite[Lemma 2.7]{Q2}, so that we have an orthogonal decomposition and projection
	\[ \omega : T \widetilde{M} = T \mathcal{O} \oplus T \mathcal{O}^\perp \rightarrow T \mathcal{O} \cong \widetilde{M} \times \g, \]
so that $\omega$ can be thougt as a $\g$-valued differentiable 1-form in $\widetilde{M}$ and define $\Omega := d \omega_{|\Lambda^2 T \mathcal{O}^\perp}$.

\begin{prop}\label{P4.2}
For every $x \in U$, consider the $\g$-module structures in $\mathcal{H}$, $T_x \widetilde{M}$, $T_x \mathcal{O}$ and $T_x \mathcal{O}^\perp$ such that $ev_x$ is a $\g$-homomorphism as in Remark \ref{R4.1}, then 
	\begin{enumerate}
\item There is a decomposition $\mathcal{H} = \mathcal{G}(x) \oplus \mathcal{H}_0(x) \oplus \mathcal{V}(x)$ of $\mathcal{G}(x)$-submodules, and 
	\[ ev_x(\mathcal{G}(x)) = T_x \mathcal{O}, \qquad ev_x(\mathcal{V}(x)) = T_x \mathcal{O}^\perp, \qquad \mathcal{H}_0(x) = Ker(ev_x).\]

\item The maps $\omega_x : T_x \widetilde{M} \rightarrow \g$ and $\Omega_x : \Lambda^2 T_x \mathcal{O}^\perp \rightarrow \g$ are $\g$-homomorphisms when $\g$ is considered as a $\g$-module with respect to the adjoint representation.

\item The restriction of the isotropy representation $\lambda_x$ gives an injective homomorphism $\lambda_0 : \mathcal{H}_0(x) \hookrightarrow \mathfrak{so}(T_x \mathcal{O}^\perp)$ such that $\Omega_x$ is $\lambda_0(\mathcal{H}_0(x))$-invariant under the identification $\mathfrak{so}(T_x \mathcal{O}^\perp) \cong \Lambda^2 T_x \mathcal{O}^\perp$, more precisely
	\[ [\lambda_0(\mathcal{H}_0(x)), \mathfrak{so}(T_x \mathcal{O}^\perp)] \subset Ker(\Omega_x), \]
where $[\cdot , \cdot ] $ is the Lie bracket of the algebra $\mathfrak{so}(T_x \mathcal{O}^\perp)$.

\item If $T_x \mathcal{O}^\perp$ is a trivial $\g$-module for every $x$ in an open subset of $U$, then $T \mathcal{O}^\perp$ is integrable.
	\end{enumerate}
\end{prop}

\begin{proof}
The proof of parts (1)-(3) are exactly as in Proposition 3.7 and Proposition 3.10 in in \cite{Q-O2}. Part (4) follows from the fact that if $T_x \mathcal{O}^\perp$ is a trivial $\g$-module, then $\Omega_x(\Lambda^2 T_x \mathcal{O}^\perp)$ is a trivial submodule of the irreducible module $\g$, so it must be the zero module. Now if $\Omega_x = 0$ for every $x$ in an open subset, then $\Omega \equiv 0$ because $\Omega$ is an analytic diferentiable form and this is equivalent to the fact that $T \mathcal{O}^\perp$ is an integrable distribution.
\end{proof}

\subsection{Structure of the centralizer of the action}

In this subsection we suppose that $dim(M) \leq n(2n+5)$, where $n = k+l \geq 3$ and $G = Sp(k,l)$. Fix a point $x \in U$ as in Remark \ref{R4.1}, consider the $\g$-module structures in $T_x M$ and $\mathcal{H}$ and recall the isotypic decomposition $\mathcal{H} = \mathcal{G}(x) \oplus \mathcal{H}_0(x) \oplus \mathcal{V}(x)$ together with the properties given by Proposition \ref{P4.2}, we give three technical Lemmas on the structure of such decomposition. The following remark is an immediate consecuence of the decomposition of representations into irreducible ones and Lemma \ref{Schur}.

\begin{remark}\label{R4.2}
If a Lie algebra $\h$ is considered as a $\g$-module for some Lie subalgebra $\g \leq \h$, and $A, V \subset \h$ are $\g$-submodules, then as a consequence of Jacobi identity, the Lie bracket induces $\g$-homomorphisms
	\[ [ \cdot , \cdot ] : A \otimes V \rightarrow \h, \qquad [ \cdot , \cdot ] : \Lambda^2 V \rightarrow \h. \]
Moreover, if $\h = W \oplus Z$ is a $\g$-module decomposition such that $Hom_\g(W , A \otimes V) = Hom_\g(W , \Lambda^2 V) = 0$, then $[A,V], [V,V] \subset Z$.
\end{remark}

\begin{lema}\label{L4.1}
If $\mathcal{V}(x) \subset \mathcal{H}$ is not a trivial $\mathcal{G}(x)$-module then $\mathcal{H}_0(x)$ is a trivial $\mathcal{G}(x)$-module with $dim(\mathcal{H}_0(x)) \leq 3$. If $\mathcal{S} = \mathcal{G}(x) \oplus L \oplus \mathcal{V}(x)$ is a Lie subalgebra of $\mathcal{H}$ with $L$ some trivial $\mathcal{G}(x)$-module then $(\mathcal{S} , \mathcal{G}(x) \oplus L )$ is a symmetric pair, in particular this is true for $\mathcal{S} = \mathcal{H}$. Moreover if $\mathcal{S}$ is a simple Lie algebra, then $L = \mathcal{H}_0(x) \cong \mathfrak{sp}(1)$ and $\mathcal{S} = \mathcal{H}$ is isomorphic to either $\mathfrak{sp}(k+1,l)$ or $\mathfrak{sp}(k,l+1)$.
\end{lema}

\begin{proof}
By the hypothesis on the dimension of $\widetilde{M}$, we have the bound $dim(\mathcal{V}(x)) \leq 4n$, so if $\mathcal{V}(x)$ is not a trivial $\g$-module, Corollary \ref{C1.4} and part 3 of Proposition \ref{P4.2} imply that $\mathcal{V}(x) \cong \mathbb{R}^{4k,4l}$, we have an isomorphism
	\[ \Lambda^2 \mathcal{V}(x) \cong \g \oplus \mathfrak{sp}(1) \oplus \bigoplus_{1 \leq i \leq 3} V_i, \quad V_i \otimes \mathbb{C} \cong V(\omega_2) \]
of $\g$-modules and $\lambda_0(\mathcal{H}_0(x)) \subset \mathfrak{sp}(1)$ under this identification, in particular it is a trivial $\g$-module, also Remark \ref{R4.2} tells us that
	\[ [ \mathcal{V}(x) , \mathcal{V}(x) ] \subset \mathcal{G}(x) \oplus L. \]
Now, the decomposition
	\[ L \otimes \mathcal{V}(x) \cong \bigoplus_{1 \leq i \leq l} \mathcal{V}(x), \qquad l = dim(L), \]
as a $\mathcal{G}(x)$-module and Remark \ref{R4.2} implies that $\mathcal{V}(x)$ is a stable $\mathcal{G}(x) \oplus L$-subspace with $\mathcal{G}(x) \oplus L$ a Lie subalgebra, so Remark \ref{RA.1} implies that $(\mathcal{S} , \mathcal{G}(x) \oplus L )$ is a symmetric pair. For the second part, recall that M. Berger classified all symmetric pairs $(\mathfrak{s},\mathfrak{t})$ with $\mathfrak{s}$ a simple Lie algebra in \cite{Be} and the only symmetric pairs that appears in this classification with $\mathfrak{t} = \mathfrak{sp}(k,l) \oplus L$ and $dim(L) \leq 3$ are the ones given by $\mathfrak{s}$ isomorphic to either $\mathfrak{sp}(k+1,l)$ or $\mathfrak{sp}(k,l+1)$ and $L \cong \mathfrak{sp}(1)$, finally we see that $\mathcal{S} = \mathcal{H}$ just by counting dimensions.\end{proof}

Recall that for any Lie algebra $\mathfrak{l}$ with a semisimple subalgebra $\mathfrak{t} \subset \mathfrak{l}$ there exists a semidirect product decomposition $\mathfrak{l} = \mathfrak{s} \ltimes \mathfrak{r}$, where $\mathfrak{s}$ is a semisimple Lie algebra of $\mathfrak{l}$ containing $\mathfrak{t}$ called a Levi factor and $\mathfrak{r}$ is the maximal solvable ideal of $\mathfrak{l}$ called the radical and denoted by $\mathfrak{r} = Rad(\mathfrak{l})$, such decomposition is called Levi decomposition, see \cite[Appendix B]{Kn}.

\begin{lema}\label{L4.2}
If $\mathcal{V}(x)$ is not a trivial $\mathcal{G}(x)$-module, then there are three possibilities on the structure of $\mathcal{H}$:
	\begin{enumerate}
		\item $Rad(\mathcal{H}) = \mathcal{H}_0(x) \oplus \mathcal{V}(x)$,
		
		\item $\mathcal{G}(x) \oplus \mathcal{V}(x)$ is a subalgebra of $\mathcal{H}$ with $\mathcal{V}(x)$ an abelian ideal, or

		\item There is an isomorphism of symmetric pairs 
	\[	\left(\mathcal{H}, \mathcal{G}(x) \oplus \mathcal{H}_0(x) \right) \cong ( \mathfrak{h} , \mathfrak{sp}(k,l) \oplus \mathfrak{sp}(1) ), \]
with $\mathfrak{h}$ isomorphic to either $\mathfrak{sp}(k+1,l)$ or $\mathfrak{sp}(k,l+1)$.
	\end{enumerate}
\end{lema}

\begin{proof} 
If $\mathcal{V}(x)$ is not a trivial $\mathcal{G}(x)$-module, then Lemma \ref{L4.1} tells us that $(\mathcal{H} , \mathcal{G}(x) \oplus \mathcal{H}_0(x) )$ is a symmetric pair. Take $\mathfrak{s}$ a Levi factor of $\mathcal{H}$ containing $\mathcal{G}(x)$, if 
	\[ \mathfrak{s} = \mathcal{G}(x) \oplus \mathbb{R}^s, \qquad s \leq 3, \]
as a $\mathcal{G}(x)$-module, then $\mathbb{R}^s$ is a simple ideal by Remark \ref{R4.2}, only possible when $s=0$ or $s = 3$ and $s = 0$ corresponds to the case where $Rad(\mathcal{H}) = \mathcal{H}_0(x) \oplus \mathcal{V}(x)$. If $s = 3$, then $Rad(\mathcal{H}) = \mathcal{V}(x)$ is an ideal such that $[\mathcal{V}(x), \mathcal{V}(x)] = 0$ by Corollary \ref{C1.4} and Remark \ref{R4.2}, in particular $\mathcal{G}(x) \oplus \mathcal{V}(x)$ is a Lie subalgebra with $\mathcal{V}(x)$ an abelian ideal. 
 
If on the other hand $\mathfrak{s} = \mathcal{G}(x) \oplus \mathcal{V}(x) \oplus \mathbb{R}^s$ as a $\mathcal{G}(x)$-module, then $(\mathfrak{s} , \mathcal{G}(x) \oplus \mathbb{R}^s )$ is a symmetric pair with $\mathfrak{s}$ semisimple. Suppose $I \subset \mathfrak{s}$ is a proper simple ideal not containing $\mathcal{G}(x)$, then $I$ is a trivial $\mathcal{G}(x)$-module of dimension bounded by $3$ and this is only possible when $I = \mathbb{R}^s$, but in this case $\mathcal{G}(x) \oplus \mathcal{V}(x)$ is another simple ideal and this is impossible by Lemma \ref{L4.1}. So $\mathfrak{s}$ is a simple Lie algebra such that $(\mathfrak{s} , \mathcal{G}(x) \oplus \mathbb{R}^s ) $ is a symmetric pair, again by Lemma \ref{L4.1} we have that $\mathfrak{s} = \mathcal{H}$ is isomorphic to either $\mathfrak{sp}(k+1,l)$ or $\mathfrak{sp}(k,l+1)$ and in particular $\mathcal{H}_0 \cong \mathfrak{sp}(1)$.
\end{proof}

\begin{lema}\label{L4.3}
	If the centralizer $\mathcal{H}$ is not a simple Lie algebra, then the normal distribution $T \mathcal{O}^\perp$ is integrable.
\end{lema}

\begin{proof}
By Proposition \ref{P4.2}, if $\mathcal{V}(x)$ is a trivial $\mathcal{G}(x)$-module for every $x$ in an open subset, then $T \mathcal{O}^\perp$ is integrable. Suppose that there exists a point $x \in U$ such that $\mathcal{V}(x)$ is a non-trivial $\mathcal{G}(x)$-module, so that if $\mathcal{H}$ is not a simple Lie algebra, by Lemma \ref{L4.2} there are two possibilities:

$\mathfrak{s} = \mathcal{G}(x) \oplus \mathcal{V}(x) \subset \mathcal{H}$ is a subalgebra having $\mathcal{V}(x)$ as an abelian ideal, in such a case we can take the semidirect product $S = G \ltimes V$, where $V = \mathcal{V}(x)$ is considered as an abelian Lie group and $G = Sp(k,l)$ acts in $V$ by integrating the irreducible $\mathcal{G}(x)$-module structure so that $S$ is a simply connected Lie group with Lie algebra $\mathfrak{s}$. As $\widetilde{M}$ is complete, there exists a right action of $S$ into $\widetilde{M}$ so that if we consider the orbit map
	\begin{displaymath}
		f : S \rightarrow \widetilde{M}, \qquad f(h) = x \cdot h,
	\end{displaymath}
then \begin{equation}\label{eq-4.2} df_e(X) = \psi(X), \qquad \forall \ X \in \mathfrak{s}, \end{equation}
where $\psi : \mathfrak{s} \rightarrow Kill(\widetilde{M})$ is the inclusion homomorphism. Property (\ref{eq-4.2}) tells us that
	\begin{displaymath}
		df_e (\mathcal{G}(x)) = T_x \mathcal{O}, \qquad  df_e(V) = T_x \mathcal{O}^\perp,
	\end{displaymath}
and in particular, $f$ is a local diffeomorphism around $e \in L$. For every $\eta \in S$, we denote the two diffeomorphisms induced by right multiplication in both $\widetilde{M}$ and $S$ as
	\begin{displaymath}
		R_\eta : S \rightarrow S, \qquad R_\eta : \widetilde{M} \rightarrow \widetilde{M},
	\end{displaymath}
so that we have $f = R_\eta \circ f \circ R_{\eta^{-1}}$, this can be seen directly from the definition of $f$. The isometric right $S$-action commutes with the left $G$-action because $\mathfrak{s}$ belongs to the centralizer $\mathcal{H}$ so that the $S$-action preserves the distributions $T \mathcal{O}$ and $T \mathcal{O}^\perp$, this implies that for every $g \in G$ and $w \in V$, we have
	\begin{equation}\label{eq-4.3}
		d (R_w)_x (T_x \mathcal{O}^\perp) = T_{x \cdot w} \mathcal{O}^\perp, \qquad d(R_g)_x (T_x \mathcal{O}^\perp) = T_{x \cdot g} \mathcal{O}^\perp,
	\end{equation}
on the other hand the semidirect product structure of $S$ implies that for every $g,h \in G$ and $w \in V$ we have
	\begin{equation}\label{eq-4.4}
		R_w (\{h\} \times V) = \{h\} \times V, \qquad R_g(\{h\} \times V) = \{hg\} \times V.
	\end{equation}
Relations (\ref{eq-4.3}) and (\ref{eq-4.4}) together with the right equivariance $d f_\eta = d (R_\eta)_x \circ d f_e \circ d (R_{\eta^{-1}})_\eta$ implies that there is a family of manifolds parametrized by elements in a neighborhood of $e \in G$ given by
	\begin{displaymath}
		N_g = f(\{g\} \times V), \qquad T_{f(g,w)} N_g = T_{f(g,w)} \mathcal{O}^\perp,
	\end{displaymath}
so that $\Omega \equiv 0$ in a neigborhood of $x$, and thus $T \mathcal{O}^\perp$ is integrable.

The second possibility is that $\mathcal{R} = \mathcal{H}_0(x) \oplus \mathcal{V}(x)$ is the maximal solvable ideal of $\mathcal{H}$, and then the semidirect product $H = G \ltimes R$ is the simply connected Lie group with Lie algebra $\mathcal{H}$, where $R$ is the simply connected Lie group with Lie algebra $\mathcal{R}$. Take $H_0 \leq R$ the connected analytic subgroup associated to the Lie subalgebra $\mathcal{H}_0(x)$, then $H_0$ is a closed subgroup in $R$ because $R$ is simply connected and solvable, also $R$ is closed in $H$ because $\mathcal{R}$ is an ideal in $\mathcal{H}$ (for these two facts see \cite{Che} or \cite[Pg. 152]{He}), then $H_0$ is also closed as a subgroup of $H$. Consider the isometric right action of $H$ into $\widetilde{M}$ as in the previous case and the induced orbit map in the quotient
	\[ f : H_0 \backslash H \rightarrow \widetilde{M}, \qquad f(H_0 \ h) = x \cdot h, \]
that is well defined because elements of $H_0$ fix $x$. As before we have that 
	\[ df_{H_0 e}(\mathcal{G}(x)) = T_x \mathcal{O}, \qquad df_{H_0 e}(\mathcal{V}(x)) = T_x \mathcal{O}^\perp, \]
and therefore $f$ is a local diffeomorphism around $H_0 e$, moreover $f$ is again right $H$-equivariant, so $f = R_\eta \circ f \circ R_{\eta^{-1}}$ for every $\eta \in H$. The isometric right $H$-action commutes with the left $G$-action because $\mathcal{H}$ is the centralizer of the action so that the $H$-action preserves the distributions $T \mathcal{O}$ and $T \mathcal{O}^\perp$, this implies that for every $h \in H$, we have
	\begin{equation}\label{eq-4.5}
		d (R_h)_x (T_x \mathcal{O}^\perp) = T_{x \cdot h} \mathcal{O}^\perp.
	\end{equation}

If we denote $H_0^h := h^{-1} H_0 h$, then $H_0^g \subset R$ for every $g \in G$ and we have a family of submanifolds of $H_0 \backslash H$
	\[ V_g := \{g \} \times (H_0^g) \backslash R \subset H_0 \backslash H, \qquad g \in G, \]
and the semidirect product structure of $H$ implies that for every $g,h \in G$ and $r \in R$ we have
	\begin{equation}\label{eq-4.6}
		R_r ( V_h ) = V_h, \qquad R_g( V_h ) = V_{hg}.
	\end{equation}
Relations (\ref{eq-4.5}) and (\ref{eq-4.6}) together with the right $H$-equivariance $d f_\eta = d (R_\eta)_x \circ d f_e \circ d (R_{\eta^{-1}})_\eta$ implies that there is a family of manifolds parametrized by elements in a neighborhood of $e \in G$ given by
	\begin{displaymath}
		N_g = f( V_g ), \qquad T_{ f(g,r) } N_g = T_{f(g,r)} \mathcal{O}^\perp,
	\end{displaymath}
$N_g$ is a family of integral manifolds of $T \mathcal{O}^\perp$ in a neigborhood of $x$, which implies that $\Omega \equiv 0$ in an non-empty open set and thus $T \mathcal{O}^\perp$ is integrable.
\end{proof}

We can summarize in the following Lemma the properties of the centralizer obtained so far in the case where $T \mathcal{O}^\perp$ is a non-integrable distribution.

\begin{lema}\label{L4.4}
If $T \mathcal{O}^\perp$ is non-integrable, then $\mathcal{H} \cong \h$ for $\h$ either $\mathfrak{sp}(k+1,l)$ or $\mathfrak{sp}(k,l+1)$. Moreover, if $U$ is as in Remark \ref{R4.1}, for every $x \in U \subset \widetilde{M}$ it is possible to choose an isomorphism $\Psi : \mathcal{H} \rightarrow \h$ so that 
	\[ \Psi(\mathcal{G}(x) ) = \mathfrak{sp}(k,l), \quad \Psi(\mathcal{H}_0(x) ) = \mathfrak{sp}(1), \quad \Psi(\mathcal{V}(x) ) = \mathbb{H}^{k,l}, \]
where the decomposition $\h = \mathfrak{sp}(k,l) \oplus \mathfrak{sp}(1) \oplus \mathbb{H}^{k,l}$ is the one given in Proposition \ref{PA.1}.
\end{lema}

\begin{proof}
If $T \mathcal{O}^\perp$ is non-integrable, then by Lemma \ref{L4.3}, $\mathcal{H}$ is a simple Lie algebra. Take a point $x \in U$ and observe that if $\mathcal{V}(x)$ is a trivial $\mathcal{G}(x)$-module, then $\mathcal{H}_0(x) \oplus \mathcal{V}(x)$ is a trivial module as well, this is a consequence of Proposition \ref{P4.2} and Remark \ref{R4.2}, which implies that $\mathcal{H}_0(x) \oplus \mathcal{V}(x)$ is an ideal of $\mathcal{H}$, an impossibility because $\mathcal{H}$ is simple. Now Lemma \ref{L4.2} applies and simplicity of $\mathcal{H}$ implies that there is an isomorphism of symmetric pairs
	\[	\left(\mathcal{H}, \mathcal{G}(x) \oplus \mathcal{H}_0(x) \right) \cong ( \mathfrak{h} , \mathfrak{sp}(k,l) \oplus \mathfrak{sp}(1) ), \]
with $\mathfrak{h}$ either $\mathfrak{sp}(k+1,l)$ or $\mathfrak{sp}(k,l+1)$, this for every $x \in U$, the last assertion follows from Theorem \ref{LA.2}.
\end{proof}

\subsection{Proof of Theorem \ref{main4}}

Lemma \ref{L4.4} tells us that an important feature that determines the structure of the centralizer algebra $\mathcal{H}$ is whether the normal distribution $T \mathcal{O}^\perp$ is integrable or not. The case when $T \mathcal{O}^\perp$ is integrable has already been studied and we have the following result that is a particular case of Theorem 1.1 in  \cite{Q2}

\begin{teo}\label{T4.2} If $T \mathcal{O}^\perp$ is integrable then there exist:
	
\begin{enumerate}
\item an isometric finite covering map $\widehat{M} \rightarrow M$ to which the $G$-action lifts,

\item a simply connected complete pseudo-Riemannian manifold $\widetilde{N}$, and

\item a discrete subgroup $\Gamma \subset G \times Iso(\widetilde{N})$,
	\end{enumerate}
such that $\widehat{M}$ is $G$-equivariantly isometric to $\left( G \times \widetilde{N} \right) / \Gamma$.
\end{teo}

\begin{lema}\label{L4.5}
If $\overline{g}$ is a metric tensor obtained by rescaling $g$ constantly along $T \mathcal{O}$ and $T \mathcal{O}^\perp$, then $Vol(M,\overline{g}) < \infty$ if and only if $Vol(M,g) < \infty$.
\end{lema}

\begin{proof}
By hypothesis, there exist two non-zero constants $c_1$ and $c_2$ such that the metric tensor $\overline{g}$ is obtained from $g$ by rescaling it along $T \mathcal{O}$ by $c_1$ and along $T \mathcal{O}^\perp$ by $c_2$. Now the volume form may be obtained from the metric tensor in local coordinates involving its determinant \cite[Lemma 7.19]{O}, so that we get
	\[ Vol_{\overline{g}} = \sqrt{|c_1|^r |c_2|^{N-r}} Vol_{g} \]
where $r = n(2n+1)$ and $N = n(2n+5)$. The result follows from the fact that constant rescalings on the volume form preserves finite volume.
\end{proof}

If the normal distribution $T \mathcal{O}^\perp$ is integrable, then Theorem \ref{T4.2} gives the first part of the theorem. Let us consider then the case where $T \mathcal{O}^\perp$ is non-integrable and fix a point $x \in U \subset \widetilde{M}$. If $H$ is the simply connected Lie group with Lie algebra $\mathcal{H}$, then by Lemma \ref{L4.4}, $H$ is isomorphic to either $Sp(k+1,l)$ or $Sp(k,l+1)$ and the pair $G \times H_0 \cong Sp(k,l) \times Sp(1)$ corresponds to the Lie subalgebra $\mathcal{G}(x) \oplus \mathcal{H}_0(x)$ with the adjoint representation in $\mathcal{V}(x)$ isomorphic to $\mathbb{R}^{4k,4l}$ considered with the natural embeddings as in Proposition \ref{PA.1}. As before, there exists an isometric right action $\widetilde{M} \times H \rightarrow \widetilde{M}$ such that
	\begin{equation}\label{eq-4.7} \frac{d}{dt}_{|_{t=0}} x \cdot exp(tX) = \Psi(X)_x, \end{equation}
where $\Psi : \mathcal{H} \rightarrow Kill(\widetilde{M})$ is the inclusion map. Consider the induced function 
	\[ f : H_0 \backslash H \rightarrow \widetilde{M}, \qquad f(H_0 h) = x \cdot h, \]
then $f$ is a well defined smooth map because elements of $H_0$ fix the point $x$. This function is equivariant with respect to the corresponding right actions, i.e. if $\eta \in H$ and 
	\[ R_\eta : \widetilde{M} \rightarrow \widetilde{M}, \quad R_\eta : H_0 \backslash H \rightarrow H_0 \backslash H, \]
are the corresponding diffeomorphisms induced by right multiplication, then $f = R_{\eta^{-1}} \circ f \circ R_\eta$. By property (\ref{eq-4.7}) we have that 
	\[ df_{H_0 e} (\mathcal{G}(x)) = T_x \mathcal{O}, \qquad df_{H_0 e} (\mathcal{V}(x)) = T_x \mathcal{O}^\perp, \]
so $f$ is a local diffeomorphism in a neighborhood of $H_0 e$, but by the $H$-equivariance and the fact that $R_\eta$ is a diffeomorphism for every $\eta \in H$, then $f$ is a local diffeomorphism everywhere. Observe now that under the identification $T_{H_0 e} H_0 \backslash H \cong \mathcal{G}(x) \oplus \mathcal{V}(x)$ (see Lemma \ref{L3.1}), then 
	\[df_{H_0 e} : \mathcal{G}(x) \oplus \mathcal{V}(x) \rightarrow T_x \widetilde{M} \]
is just the evaluation map so that by Remark \ref{R4.1}, the pull-back of the metric tensor $g_x$ under $df_{H_0 e}$ is a $\mathfrak{sp}(k,l)$-invariant bilinear form and by Lemma \ref{R1.4} we can rescale $g_x$ along $T_x \mathcal{O}$ and $T_x \mathcal{O}^\perp$ such that the bilinear form we obtain in $\mathcal{G}(x) \oplus \mathcal{V}(x)$ is the restriction of the Killing form in $\mathcal{H}$. Using this, the fact that the corresponding $R_\eta$ are isometries of the Killing form and the metric tensor $g$, and the fact that 
	\[ df_\eta = d (R_\eta)_{H_0 e} \circ df_{H_0 e} \circ d (R_{\eta^{-1}})_\eta, \qquad \forall \eta \in H, \]
we have that we can rescale the pseudo-Riemannian metric $g$ along $T \mathcal{O}$ and $T \mathcal{O}^\perp$ such that with the rescaled metric $\overline{g}$,
	\[ f : ( H_0 \backslash H , h) \rightarrow (\widetilde{M}, \overline{g}) \]
is a local isometry, where $h$ is the pseudo-Riemannian metric induced from the Killing form of $\mathcal{H}$. Now both $\widetilde{M}$ and $H_0 \backslash H$ are connected, simply connected and $H_0 \backslash H$ is complete, which implies that $f$ is in fact an isometry of pseudo-Riemannian manifolds see \cite[Corollary 7.29]{O}. So we have a pseudo-Riemannian covering map 
	\[ H_0 \backslash H \rightarrow (M,\overline{g}), \]
where $\overline{g}$ is the pseudo-Riemannian metric obtained from $g$ by rescaling along $T \mathcal{O}$ and $T \mathcal{O}^\perp$ and $H_0 \backslash H$ is considered with the metric tensor induced from the Killing form in $H$. By Lemma \ref{L4.5}, $(M, \overline{g})$ is again a finite volume pseudo-Riemannian manifold such that the $Sp(k,l)$-action in $M$ lifts to $H_0 \backslash H \cong \widetilde{M}$ commuting with the right $H$-action, so that Theorem \ref{T3.2} tells us that there exists a lattice $\Gamma \subset H$ and a continuous homomorphism $\rho : G \rightarrow G \subset H$ so that we have a pseudo-Riemannian finite covering
	\[ \widehat{M} := H_0 \backslash H / \Gamma \rightarrow M, \]
that is $G$-equivariant, where the $G$-acion in $\widehat{M}$ and $H_0 \backslash H$ is given by left multiplications of $\rho(G)$ and the Theorem follows.

%\addcontentsline{toc}{chapter}{References}

\end{document}